\documentclass[12pt,reqno]{amsart}

\usepackage{amstext}
\usepackage{amsthm}
\usepackage{amsmath}
\usepackage{amssymb}
\usepackage{latexsym}
\usepackage{amsfonts}
\usepackage{graphicx}
\usepackage{graphpap}
\usepackage{enumerate}
\allowdisplaybreaks

\usepackage[pagebackref,hypertexnames=false, colorlinks, citecolor=red, linkcolor=red]{hyperref} 
\usepackage[backrefs]{amsrefs}

\input txdtools

\begin{document}

\newcommand{\ci}[1]{_{ {}_{\scriptstyle #1}}}

\newcommand{\norm}[1]{\ensuremath{\left\|#1\right\|}}
\newcommand{\Norm}[1]{\ensuremath{\Big\|#1\Big\|}}
\newcommand{\abs}[1]{\ensuremath{\left\vert#1\right\vert}}
\newcommand{\ip}[2]{\ensuremath{\left\langle#1,#2\right\rangle}}
\newcommand{\Ip}[2]{\ensuremath{\Big\langle#1,#2\Big\rangle}}
\newcommand{\adj}[1]{#1^{*}}
\newcommand{\p}{\ensuremath{\partial}}
\newcommand{\pr}{\mathcal{P}}

\newcommand{\pbar}{\ensuremath{\bar{\partial}}}
\newcommand{\db}{\overline\partial}
\newcommand{\D}{\mathbb{D}}
\newcommand{\B}{\mathbb{B}}
\newcommand{\Sp}{\mathbb{S}}
\newcommand{\T}{\mathbb{T}}
\newcommand{\R}{\mathbb{R}}
\newcommand{\Z}{\mathbb{Z}}
\newcommand{\C}{\mathbb{C}}
\newcommand{\Cd}{\mathbb{C}^{d}}
\newcommand{\N}{\mathbb{N}}
\newcommand{\scrH}{\mathcal{H}}
\newcommand{\scrL}{\mathcal{L}}
\newcommand{\td}{\widetilde\Delta}
\newcommand{\rp}{\right)}
\newcommand{\lp}{\left(}
\newcommand{\h}{\ell^{2}}
\renewcommand{\l}[1]{\mathcal{L}{#1}}
\newcommand{\bo}{\mathcal{B}(\Om)}
\newcommand{\lh}{L_{\l(\h)}^{\infty}}
\newcommand{\lf}{L_{\textnormal{fin}}^{\infty}}
\newcommand{\m}{\mathfrak{d}}

\newcommand{\BB}{\mathcal{B}}
\newcommand{\HH}{\mathcal{H}}
\newcommand{\KK}{\mathcal{K}}
\newcommand{\LL}{\mathcal{L}}
\newcommand{\MM}{\mathcal{M}}
\newcommand{\FF}{\mathcal{F}}

\newcommand{\Om}{\Omega}
\newcommand{\La}{\Lambda}

\newcommand{\rk}{\operatorname{rk}}
\newcommand{\card}{\operatorname{card}}
\newcommand{\ran}{\operatorname{Ran}}
\newcommand{\osc}{\operatorname{OSC}}
\newcommand{\im}{\operatorname{Im}}
\newcommand{\re}{\operatorname{Re}}
\newcommand{\tr}{\operatorname{tr}}
\newcommand{\vf}{\varphi}
\newcommand{\f}[2]{\ensuremath{\frac{#1}{#2}}}

\newcommand{\kzp}{k_z^{(p,\alpha)}}
\newcommand{\klp}{k_{\lambda_i}^{(p,\alpha)}}
\newcommand{\TTp}{\mathcal{T}_p}

\newcommand{\vp}{\varphi}
\newcommand{\al}{\alpha}
\newcommand{\be}{\beta}
\newcommand{\la}{\lambda}
\newcommand{\li}{\lambda_i}
\newcommand{\lb}{\lambda_{\beta}}
\newcommand{\Bo}{\mathcal{B}(\Omega,\C)}
\newcommand{\boa}{\mathcal{B}_{\mathcal{A}}(\Omega)}
\newcommand{\Boa}{\mathcal{B}_{\mathcal{A}}(\Omega,\C)}
\newcommand{\Bbp}{\mathcal{B}_{\beta}^{p}}
\newcommand{\Bbt}{\mathcal{B}(\Omega)}
\newcommand{\Lbt}{L_{\beta}^{2}}
\newcommand{\Kz}{K_z}
\newcommand{\kz}{k_z}
\newcommand{\Kl}{K_{\lambda_i}}
\newcommand{\kl}{k_{\lambda_i}}
\newcommand{\Kw}{K_w}
\newcommand{\kw}{k_w}
\newcommand{\Kbz}{K_z}
\newcommand{\Kbl}{K_{\lambda_i}}
\newcommand{\kbz}{k_z}
\newcommand{\kbl}{k_{\lambda_i}}
\newcommand{\Kbw}{K_w}
\newcommand{\kbw}{k_w}
\newcommand{\BL}{\mathcal{L}\left(\mathcal{B}(\Omega), L^2(\Om,\h;d\sigma)\right)}
\newcommand{\fl}{\mathcal{L}}
\newcommand{\ml}{M_{I^{(d)}}}

\newcommand{\af}{\mathfrak{a}}
\newcommand{\bb}{\mathfrak{b}}
\newcommand{\cc}{\mathfrak{c}}

\newcommand{\entrylabel}[1]{\mbox{#1}\hfill}

\newenvironment{entry}
{\begin{list}{X}%
  {\renewcommand{\makelabel}{\entrylabel}%
      \setlength{\labelwidth}{55pt}%
      \setlength{\leftmargin}{\labelwidth}
      \addtolength{\leftmargin}{\labelsep}%
   }%
}%
{\end{list}}


\numberwithin{equation}{section}

\newtheorem{thm}{Theorem}[section]
\newtheorem{lm}[thm]{Lemma}
\newtheorem{cor}[thm]{Corollary}
\newtheorem{conj}[thm]{Conjecture}
\newtheorem{prob}[thm]{Problem}
\newtheorem{prop}[thm]{Proposition}
\newtheorem*{prop*}{Proposition}

\theoremstyle{remark}
\newtheorem{rem}[thm]{Remark}
\newtheorem*{rem*}{Remark}
\newtheorem{example}[thm]{Example}

\title[$\h$--valued Bergman--type function spaces]
{A Reproducing Kernel Thesis for Operators on 
$\h$--valued Bergman-type Function Spaces}

\author[R. Rahm]{Robert Rahm}
\address{Robert S. Rahm, School of Mathematics\\ Georgia Institute of Technology
\\ 686 Cherry Street\\ Atlanta, GA USA 30332-0160}
\email{rrahm3@math.gatech.edu}
\urladdr{www.math.gatech.edu/~rrahm3}

\begin{abstract}
In this paper we consider the reproducing kernel thesis for boundedness and 
compactness for operators on $\h$--valued Bergman-type spaces. This
paper generalizes many well--known results about classical function spaces
to their $\h$--valued versions. In particular, the 
results in this paper apply to the weighted $\h$--valued Bergman space on the 
unit ball, the unit polydisc and, more generally to weighted Fock spaces.  
\end{abstract}

\maketitle

\section{Introduction}

In \cite{MW2}, Mitkovski and Wick show that in a wide variety of classical 
functions spaces (they call these spaces Bergman--type function spaces), many 
properties of an operator can be determined by studying its behavior on the 
normailzed reproducing kernels. Thus, their results are ``Reproducing Kernel
Thesis'' (RKT) statements. 

The unified approach developed in \cite{MW2} was used to solve 
two types of problems relating to operators on classical function spaces:
boundedness and compactness. The goal of this paper is to extend this approach
to the case of $\h$--valued Bergman type function spaces and to prove results
relating to boundedness and compactness of operators
for a general class of $\h$--valued Bergman type function spaces. 

The proofs in
this paper are essentially the same as the corresponding proofs from \cite{MW2}.
The only adjustments are that our integrals are now vector--valued and
we must use a version of the classical Schur's test for integral operators with
matrix--valued kernels. This is Lemma~\ref{MSchur}. While this lemma is not 
deep, and is probably known (or at least expected) by experts, we were unable
to find it in the literature.

The paper is organized as follows. In Section~\hyperref[vvbts]{2}, we give a 
precise
definition of $\h$--valued Bergman--type spaces and prove some of their
basic properties. In Section~\hyperref[bdd]{3}, we prove RKT statements for 
boundedness
and extend several classical results about Toeplitz and Hankel operators to the 
$\h$--valued setting. In Section~\hyperref[RKTComp]{4}, we prove RKT statements for 
compactness. In the final section, Section~\hyperref[dens]{5}, 
we show that an operator is compact if and 
only if 
it is in the Toeplitz algebra and its Berezin transform vanishes on the 
boundary of $\Om$.

\section{$\h$--Valued Bergman-type spaces}\label{vvbts}
Before we can define the $\h$--valued Bergman--type spaces, we will need to
make some general definitions regrading $\h$--valued functions.
Let $\Om$ be a domain (connected open set) in $\mathbb{C}^{n}$, let 
$\mu$ be a measure on $\Om$ and let 
$\{e_k\}_{k=1}^{\infty}$ be the standard orthonormal basis for $\h$. 
We say that function
$f:\Om\to\h$ is $\mu$--measurable (analytic) if for each $k\in\N$ the function 
$z\mapsto \ip{f(z)}{e_k}_{\h}$ is $\mu$--measurable (analytic) on $\Om$. For
a $\mu$--measurable set $E\subset\Om$, and a $\mu$--measurable function $f$,
we define the integral of $f$ over $E$:
\begin{align*}
\int_{E}fd\mu := \sum_{k=1}^{\infty}
    \left(\int_{E}\ip{f}{e_k}_{\h}d\mu\right)e_k.
\end{align*}
This is a well--defined element of $\h$ whenever
\begin{align*}
\sum_{k=1}^{\infty}\abs{\int_{E}\ip{f}{e_k}_{\h}d\mu}^{2}<\infty.
\end{align*}
The space $L^{2}(\Om,\C;\mu)$ is the space of all $\C$--valued 
$\mu$--measurable functions, $g$, such that
\begin{align*}
\norm{g}_{L^{2}(\Om,\C;\mu)}^{2}:=\int_{\Om}\abs{g}^{2}d\mu <\infty. 
\end{align*}
The space $L^{2}(\Om,\h;\mu)$ is the space of all $\h$--valued measurable
functions, $f$, such that 
\begin{align*}
\norm{f}_{L^{2}(\Om,\h;\mu)}^{2}
:=\int_{\Om}\norm{f}_{\h}^{2}d\mu.
\end{align*}

\noindent Note that $L^{2}(\Om,\h;\mu)$ is a Hilbert space with inner product:
\begin{align*}
\ip{f}{g}_{L^{2}(\Om,\h;\mu)}:=\int_{\Om}\ip{f}{g}_{\h}d\mu
\end{align*} 
and in this case
\begin{align*}
\norm{f}_{L^{2}(\Om,\h;\mu)}^{2}
=\sum_{k=1}^{\infty}\int_{\Om}\abs{\ip{f}{e_k}_{\h}}^{2}d\mu
=\sum_{k=1}^{\infty}\norm{\ip{f}{e_k}}_{L^{2}(\Om,\C;d\mu)}^{2},
\end{align*}
where $\ip{\cdot}{\cdot}_{\h}$ is the standard inner product on $\h$.
The spaces we will consider in this paper are spaces of functions that
take values in $\h$. However, we will also have occasion to discuss some spaces 
of $\ell^{p}$--valued functions. We will refer to such spaces as 
``vector--valued function spaces''. 

The space $L^{p}(\Om,\C;\mu)$ is the space of all $\C$--valued 
$\mu$--measurable functions, $g$, such that
\begin{align*}
\norm{g}_{L^{p}(\Om,\C;\mu)}^{p}:=\int_{\Om}\abs{g}^{p}d\mu <\infty. 
\end{align*}
The space $L^{p}(\Om,\ell^p;\mu)$ is the space of all $\ell^p$--valued 
measurable functions, $f$, such that 
\begin{align*}
\norm{f}_{L^{p}(\Om,\h;\mu)}^{p}
:=\int_{\Om}\norm{f}_{\ell^p}^{p}d\mu
=\sum_{k=1}^{\infty}\int_{\Om}\abs{\ip{f}{e_k}_{\h}}^{p}d\mu.
\end{align*}

The functions $\norm{\cdot}_{L^{p}(\Om,\ell^p;\mu)}$ clearly satisfy 
$\norm{\lambda f}_{L^{p}(\Om,\ell^p;\mu)}=
\lambda\norm{f}_{L^{p}(\Om,\ell^p;\mu)}$
and the triangle inequality.
If we identify two functions if $\norm{f(z)-g(z)}_{\ell^p}=0$ for $\mu$-a.e. 
$z\in\Om$ then $\norm{\cdot}_{L^{p}(\Om,\ell^p;\mu)}$ is positive definite.
Therefore, the functions  $\norm{\cdot}_{L^{p}(\Om,\ell^p;\mu)}$ define
norms. The spaces are also complete (see, for example, \cite{G}) and so they
are all Banach spaces.

We introduce a large class of 
$\h$--valued reproducing kernel Hilbert spaces that will form 
an abstract framework for our results. Due to their similarities with the 
classical Bergman space we call them 
$\h$--valued Bergman-type spaces.  In defining the key 
properties of these spaces, we use the standard notation that $A\lesssim B$ to 
denote that there exists a constant $C$ such that $A\leq C B$.  And, $A\simeq B$
which means that $A\lesssim B$ and $B\lesssim A$.

Below we list the defining properties of these spaces.

\begin{itemize}

\item[\label{A1} A.1] Let $\Om$ be a domain (connected open set) in $\C^n$ which
contains the origin. We assume that for each $z\in\Om$, there exists an 
involution $\varphi_z \in \textnormal{Aut}(\Om)$ satisfying $\varphi_z(0)=z$. 

\item[\label{A2} A.2] We assume the existence of a metric $\m$ on $\Omega$ 
which is quasi-invariant under $\vp_z$, i.e., $\m(u,v)\simeq 
\m(\vp_z(u),\vp_z(v))$ with the implied constants independent
of $u,v\in\Om$. In addition, we assume that the metric space $(\Om, \m)$ is 
separable and finitely compact, i.e., every closed ball in $(\Om, \m)$ is 
compact. As usual, we denote by $D(z, r)$ the disc centered at $z$ with radius 
$r$ with respect to the metric $\m$. 

\item[\label{A3} A.3] We assume the existence of a finite Borel measure $\sigma$
on $\Om$ and define $\bo$ to be the space of $\h$--valued analytic functions 
on $\Om$ equipped with the $L^2(\Om,\h;d\sigma)$ norm. 
We shall also have occasion to consider the 
space of $\C$--valued analytic functions on $\Om$ that are also in 
$L^{2}(\Om,\C;d\sigma)$. We will denote this space by $\Bo$. Note that this
space is the ``scalar--valued'' Bergman--type space as defined in \cite{MW2}.
Everywhere in the paper, $\norm{\,\cdot\,}_{\bo}$ 
and $\ip{\,\cdot}{\cdot\,}_{\bo}$ will denote the norm and the inner 
product in $L^2(\Om,\h;d\sigma)$ and $\norm{\,\cdot\,}_{\Bo}$ 
and $\ip{\,\cdot}{\cdot\,}_{\Bo}$ will always denote the norm and the inner 
product in $L^2(\Om,\C;d\sigma)$. We assume that $\bo$ is a reproducing kernel 
Hilbert space
(RKHS) and denote by $\Kz$ and $\kz$ the reproducing and the normalized 
reproducing kernels in $\Bo$. That is, for every $g\in\Bo$, and every $z\in\Om$
there holds:
\begin{align*}
g(z)=\int_{\Om}\overline{K_z(w)}g(w)d\sigma(w).
\end{align*}
And for every $f\in\bo$ and $z\in\Om$, there holds:
\begin{align*}
f(z)=\int_{\Om}\overline{K_z(w)}f(w)d\sigma(w)
=\int_{\Om}\ip{K_w}{K_z}_{\Bo}f(w)d\sigma(w).
\end{align*}
To emphasize, the reproducing kernels $K_z$ are $\C$--valued analytic functions
in $\Bo$ and they act as reproducing kernels on both spaces $\bo$ and $\Bo$.
We will also assume that $\norm{K_z}_{\Bo}$  is continuous 
as a function of $z$ taking $(\Om, \m)$ into $\R$.
\end{itemize}

We will say that $\bo$ is an \textit{$\h$--valued 
Bergman-type space} if in addition to 
\hyperref[A1]{A.1}-\hyperref[A3]{A.3} it also satisfies the following 
properties.

\begin{itemize}
\item[\label{A4} A.4] We assume that the measure 
$d\la(z):=\norm{K_z}_{\Bo}^2d\sigma(z)$ is quasi-invariant under all $\vp_z$, 
i.e., for every Borel set $E\subset\Om$  we have $\la(E)\simeq \la(\vp_z(E))$ 
with the implied constants independent of $z\in\Om$. In addition, we assume that
$\la$ is doubling, i.e., there exists a constant $C>1$ such that for all 
$z\in\Om$ and $r>0$ we have $\la(D(z,2r))\leq C\la(D(z,r))$.

\item[\label{A5} A.5] We assume that 
$$\abs{\ip{k_z}{k_w}_{\Bo}}\simeq \frac{1}{\norm{K_{\vp_z(w)}}_{\Bo}},$$ 
with the implied constants independent of $z, w\in\Om$. 

\item[\label{A6} A.6] We assume that there exists a positive constant  
$\kappa<2$ such that 
\begin{equation} \label{propA6}
\int_{\Om}{\frac{\abs{\ip{K_z}{K_w}_{\Bo}}^{\frac{r+s}{2}}}
{\norm{K_z}_{\Bo}^s\norm{K_w}_{\Bo}^r}\,d\la(w)}\leq C = C(r,s) 
< \infty, \; \; \forall z \in \Om 
\end{equation}
for all $r>\kappa>s>0$ or that~\eqref{propA6} holds for all $r=s>0$.   In the 
latter case we will say that $\kappa =0$. These will be called the Rudin-Forelli
estimates for $\Bo$.

\item[\label{A7} A.7] We assume that $\lim_{\m(z,0)\to\infty}
\norm{K_z}_{\Bo}=\infty$.

\end{itemize}

We say that $\bo$ is a \textit{strong $\h$--valued 
Bergman-type space} if we have $=$ instead of $\simeq$ everywhere in 
\hyperref[A1]{A.1}-\hyperref[A5]{A.5}. 

\subsection{Some Examples}
The classical Bergman spaces on the unit ball, polydisc, or over
any bounded symmetric domain that satisfies the Rudin--Forelli estimates are
all examples of scalar--valued Bergman--type spaces. It should be 
pointed out that in classical Bergman spaces on the ball, 
the invariant measure of \hyperref[A4]{A.4} is not strictly
doubling. However, the only place where the doubling property 
is used is in geometric decomposition of
$\Om$ in Proposition \ref{Covering}. However, results of this type
are well known for the classical Bergman spaces on the ball.
See for example \cites{MW2,CR,Sua,MSW,BI}.

Additionally, the classical
Fock space is a scalar--valued Bergman--type space. For a more 
detailed discussion of examples of Bergman--type spaces, 
see \cite{MW2}. Clearly, any Bergman--type space can be extended to a 
$\h$--valued Bergman space and so the $\h$--valued versions of these
spaces are $\h$--valued Bergman--type spaces.

\subsection{Classical Results Extended to the $\h$--Valued Setting}
Before going on, we discuss notation. If $\mathcal{X}$ and $\mathcal{Y}$ are
Banach spaces, $\mathcal{L}(\mathcal{X},\mathcal{Y})$ is the space of bounded
linear operators from $\mathcal{X}$ to $\mathcal{Y}$ equiped with the
usual operator norm. When $\mathcal{X}=
\mathcal{Y}$ we will write $\mathcal{L}(\mathcal{X},\mathcal{Y})=
\mathcal{L}(\mathcal{X})$. 

The symbols $\norm{\cdot}$ and $\ip{\cdot}{\cdot}$ will
be used in several different ways throughout the paper. To make things clear,
we will adorn these symbols with a subscript to indicate the space in which the
norm or inner product is being taken.

For the rest of the paper, let $\{e_k\}_{k=1}^{\infty}$ denote the standard 
orthonormal basis for $\h$. If $v$ is an element of $\h$, then $v_k$ will 
denote the $k^{th}$ component of $v$. That is $v_k=\ip{v}{e_k}_{\h}$.
Similarly, if $f$ is an $\h$--valued function, $f_k$ will denote the 
$k^{th}$ component function. That is, $f_k(z)=\ip{f(z)}{e_k}_{\h}$.

The identity operator on $\h$ will be denoted by $I$. 
In addition, if $d\in\N$, $I^{(d)}$ is the operator that is the 
orthogonal projection onto the span of $\{e_1,\cdots,e_d\}$. That is, 
$I^{(d)}$ is the identity matrix with the first $d$ entries on the diagonal 
set equal to $1$ and all other entries set to $0$. Also, $I_{(d)}$ will be the
``opposite'' of $I^{(d)}$. That is, $I_{(d)}=I-I^{(d)}$. 

If $e\in\h$, we say that $e$ is $d$--finite if there is a $d\in\N$ such that
$e=I^{(d)}e$. That is, only the first $d$ entries of $e$ may be non--zero.
An operator $U\in\l(\h)$ will be 
called $d$--finite if there is a $d\in\N$ such that 
$U=I^{(d)}UI^{(d)}$. Equivalently, $\ip{Ue_i}{e_k}_{\h}=0$ if either 
$i>d$ or $k>d$. An $\h$--valued function $f:\Om\to\h$ will be called 
$d$--finite if $f(z)$ is $d$--finite for all $z\in\Om$. A matrix--valued
function $U:\Om\to\l(\h)$ will be called $d$--finite if $U(z)$ is a 
$d$--finite operator on $\h$ for every $z\in\Om$. If the exact value
of $d$ is not important, we will simply say ``finite'' instead of 
$d$--finite. For example, a vector $u\in\h$ is finite if there is a
$d$ such that $u$ is $d$--finite.

We will often refer to linear operators on $\h$ (not neccessairily bounded)
as matricies. There should be no confusion that these 
matricies are infinite dimensional matricies and are written relative 
to the standard orthonormal basis $\{e_k\}_{k=1}^{\infty}$.

Let $\mathcal{C}=\{f_{a}\}_{a\in\mathcal{A}}$ be a collection of $\C$--valued 
functions. A linear combination of functions in the collection 
$\{f_{a}\}_{a\in\mathcal{A}}$ is a sum of the form:
\begin{align}\label{lincom}
f_1h_1 + \cdots + f_mh_m,
\end{align}
where each $f_i\in\mathcal{C}$, each $h_i$ is a finite element of $\h$ and
$m<\infty$. A $\C$--linear combination of functions in the collection
is a sum of the form:
\begin{align*}
f_1c_1 + \cdots +f_mc_m,
\end{align*}
where the $c_i$ are complex numbers. To reiterate, whenever we say 
``linear combination'', we will mean one as defined in \eqref{lincom} so that
a linear combination of scalar--valued functions is an $\h$--valued function.

\begin{lm}
Let $\mathcal{C}$ be a collection of $\C$--valued functions such that the
set of $\C$--linear combinations of functions in $\mathcal{C}$ is dense in
$\Bo$. Then the linear combinations of elements of $\mathcal{C}$ is dense
in $\bo$.
\end{lm}
\begin{proof}
First, let $g\in\bo$ be finite. Then since the $\C$--linear combinations of
elements of
$\mathcal{C}$ are dense in $\Bo$, we can approximate $g$ in the 
$\bo$ norm with linear combinations of elements of $\mathcal{C}$.
Let $f\in\bo$ be arbitary. Then 
$\norm{f}_{\bo}^{2}=\sum_{k}\norm{\ip{f}{e_k}_{\h}}^{2}_{\Bo}<\infty$. 
Thus there is an $N\in\N$ such that 
$\sum_{k=N}^{\infty}\norm{\ip{f}{e_k}_{\h}}^2_{\Bo}<\epsilon$ and so
$\norm{f-\sum_{k=1}^{N-1}\ip{f}{e_k}}_{\bo}<\epsilon$. That is, 
$f$ can be approximated by finite elements of $\bo$ in the $\bo$ norm. 
Let $g$ be a finite element
of $\bo$ such that $\norm{f-g}_{\bo}\leq \epsilon$ and let $h$ be a linear 
combination of elements of $\mathcal{C}$ such that $\norm{g-h}_{\bo}\leq\epsilon$.
Then there holds:
\begin{align*}
\norm{f-h}_{\bo}\leq\norm{f-g}_{\bo}+\norm{h-g}_{\bo} \leq 2\epsilon.
\end{align*}
This completes the proof. 
\end{proof}
\noindent This implies the following corollary:
\begin{cor}
The linear combinations of the normalized reproducing kernels, reproducing 
kernels, and monomials are all dense in $\bo$.
\end{cor}

\subsection{Projection Operators on Bergman-type Spaces} 
It is easy to see that the orthogonal projection of $L^2(\Om,\h;d\sigma)$ onto 
$\bo$ is given by the integral operator
$$
P(f)(z):=\int_{\Om}\ip{\Kbw}{\Kbz}_{\Bo}f(w)d\sigma(w).
$$
Therefore, for all $f\in\bo$ we have 
$f(z)=\int_{\Om}\kbw(z)\ip{f}{\kbw}\,d\lambda(w).$
Moreover, 
$$
\norm{f}_{\bo}^2
=\int_{\Om}\ip{f(w)}{f(w)}_{\h}d\sigma(w)
=\int_{\Om}\sum_{j=1}^{\infty}\abs{\ip{f_j}{k_w}_{\h}}^2d\la(w).
$$

If $\kappa>0$, $P$ is bounded as an operator on 
$L^p(\Om,\ell^p;d\sigma)$ for $1<p<\infty$ and if $\kappa=0$, $P$ is bounded as an 
operator on $L^p\left(\Om,\ell^p;\frac{d\sigma(w)}{\norm{K_w}^p}_{\Bo}\right)$. In 
\cite{MW2}, 
the authors prove:
\begin{lm}\label{projs} Let $P(f)(z)
:=\int_{\Om}\ip{\Kbw}{\Kbz}_{\Bo}f(w)d\sigma(w)$ 
be the projection operator on $L^{p}(\Om,\C;d\sigma)$.
\begin{enumerate}[\textnormal{(}a\textnormal{)}]
\item If $\kappa=0$ then $P$ is bounded as an operator from 
$L^p\left(\Om,\C;\frac{d\lambda(w)}{\norm{K_w}_{\Bo}^p}\right)$ into 
$L^p\left(\Om,\C;\frac{d\lambda(w)}{\norm{K_w}_{\Bo}^p}\right)$ for all 
$1\leq p\leq \infty$. 
\item If $\kappa>0$ then $P$ is bounded as an operator from 
$L^p(\Om,\C;d\sigma)$ 
into $L^p(\Om,\C;d\sigma)$ for all $1<p<\infty$
\end{enumerate}
\end{lm}

The proof of the following lemma is easily deduced
from Lemma \ref{projs} and is omitted. 

\begin{lm}\label{proj} Let $P(f)(z)
:=\int_{\Om}\ip{\Kbw}{\Kbz}_{\Bo}f(w)d\sigma(w)$ 
be the projection operator.
\begin{enumerate}[\textnormal{(}a\textnormal{)}]
\item If $\kappa=0$ then $P$ is bounded as an operator from 
$L^p\left(\Om,\ell^p;\frac{d\lambda(w)}{\norm{K_w}_{\Bo}^p}\right)$ into 
$L^p\left(\Om,\ell^p;\frac{d\lambda(w)}{\norm{K_w}_{\Bo}^p}\right)$ for all 
$1\leq p\leq \infty$. 
\item If $\kappa>0$ then $P$ is bounded as an operator from 
$L^p(\Om,\ell^p;d\sigma)$ 
into $L^p\left(\Om,\ell^p;d\sigma\right)$ for all $1<p<\infty$
\end{enumerate}
\end{lm}

The following is a matrix version of the classical Schur's Test. 
\begin{lm}[Schur's Test for Matrix--Valued Kernels]
\label{MSchur}
Let $(X,\mu)$ and $(X,\nu)$ be measure spaces and $M(x,y)$ a
measureable matrix--valued function on $X\times X$ whose entries are 
non--negative. That is, for all $k,i\in\N$ there holds:
\begin{align*}
\ip{M(x,y)e_k}{e_i}_{\h}\geq 0.
\end{align*}
If $h$ is a positive measureable function (with respect to
$\mu$ and $\nu$), and if $C_1,C_2$ are positive constants such that
\begin{align*}
\int_{X}\sum_{k=1}^{\infty}h(y)^{q}\ip{M(x,y)e_k}{e_i}_{\h}d\nu(y)
\leq C_1h(x)^q
\textnormal{ for } \mu\textnormal{-almost every } x;
\\
\int_{X}\sum_{i=1}^{\infty}h(x)^{p}\ip{M^*(x,y)e_i}{e_k}_{\h}d\mu(x)
\leq C_2h(y)^p
\textnormal{ for } \nu\textnormal{-almost every } y,
\end{align*}
then $Tf(x)=\int_{X}M(x,y)f(y)d\nu(y)$ defines a bounded operator 
 $T:L^{p}(X,\ell^p;\nu)\to L^{p}(X,\ell^p;\mu)$ with norm no greater than
than $C_1^{1/q}C_2^{1/p}$.
\end{lm}
\begin{proof}
The proof is simply an appropriate adaptation of a standard proof for the
classical Schur's Test. 
The following computation 
uses H\"{o}lder's Inequality at the level of the integral and at 
the level of the infinite sum, we also use the first assumption:
\begin{align*}
\abs{(Tf_i)(x)}
&=\abs{\ip{Tf(x)}{e_i}_{\h}}
\\&\leq\int_{X}\sum_{k=1}^{\infty}h(y)h(y)^{-1}
    \abs{f_k(y)}\ip{M(x,y)e_k}{e_i}_{\h}d\nu(y)
\\&\leq\int_{X}
    \left\{\sum_{k=1}^{\infty}h^{q}(y)
    \ip{M(x,y)e_k}{e_i}_{\h}\right\}^{\frac{1}{q}}
    \left\{\sum_{k=1}^{\infty}h^{-p}(y)\abs{f_k(y)}^{p}
    \ip{M(x,y)e_k}{e_i}_{\h}\right\}^{\frac{1}{p}}d\nu(y)
\\&\leq\left\{\int_{X}\sum_{k=1}^{\infty}h(y)^{q}
    \ip{M(x,y)e_k}{e_i}_{\h}d\nu(y)\right\}^{\frac{1}{q}}
    \left\{\int_{X}\sum_{k=1}^{\infty}h^{-p}(y)\abs{f_k(y)}^{p}
    \ip{M(x,y)e_k}{e_i}_{\h}d\nu(y)\right\}^{\frac{1}{p}}
\\&\leq C_1^{\frac{1}{q}}h(x)\left\{\sum_{k=1}^{\infty}\int_{X}h^{-p}(y)\abs{f_k(y)}^{p}
    \ip{M(x,y)e_k}{e_i}_{\h}d\nu(y)\right\}^{\frac{1}{p}}.
\end{align*}
Using the above estimate and the second assumption, there holds:
\begin{align*}
\norm{Tf}_{L^{p}(X,\ell^p;\mu)}^p 
&=\int\sum_{i=1}^{\infty}\abs{\ip{Tf(x)}{e_i}_{\h}}^{p}d\mu(x)
\\&\leq \int_{X}\sum_{i=1}^{\infty}\left\{C_1^{\frac{1}{q}}h(x)
    \left(\int_{X}\sum_{k=1}^{\infty}h^{-p}(y)
    \abs{f_k(y)}^{p}\ip{M(x,y)e_k}{e_i}_{\h}d\nu(y)
    \right)^{\frac{1}{p}}\right\}^{p}d\mu(x)
\\&=C_1^{\frac{p}{q}}\int_{X}\sum_{k=1}^{\infty}\abs{f_k(y)}^{p}h^{-p}(y)
    \int_{X}\sum_{i=1}^{\infty}h^p(x)\ip{M^{*}(x,y)e_i}{e_k}_{\h}d\mu(x)d\nu(y)
\\&\leq C_1^{\frac{p}{q}}C_2\int_{X}\sum_{k=1}^{\infty}\abs{f_k(y)}^{p}d\nu(y)
\\&=C_1^{\frac{p}{q}}C_2\norm{f}_{L^{p}(X,\ell^p;\nu)}^{p}.
\end{align*}
Now take $p^{th}$ roots.
The interchange of integrals and sums and the switching the order of integration
are justified since the integrand is non--negative.
\end{proof}

The following result will be useful later when applying the 
Matrix Schur's Test, Lemma \ref{MSchur}. See \cite{MW2} for
the proof.
\begin{lm}\label{RF} For all $r, s\in\R$ the following quasi-identity holds 
\begin{equation}
\int_{\Om} {\frac{\abs{\ip{K_z}{K_w}_{\Bo}}^{\frac{r-s}{2}}}
    {\norm{K_w}_{\Bo}^r}\,d\la(w)}
\simeq\int_{\Om} {\frac{\abs{\ip{K_z}{K_w}_{\Bo}}^{\frac{r+s}{2}}}
    {\norm{K_z}_{\Bo}^s\norm{K_w}_{\Bo}^r}\,d\la(w)}
\end{equation}
where the implied constants are independent of $z\in\Om$ and may depend on 
$r,s$.
\end{lm}

\subsection{Translation Operators on Bergman-type Spaces}
 
For each $z\in\Om$ we define an adapted translation operator $U_z$ on  $\bo$ by
$$ 
U_zf(w):=f(\vp_z(w))k_z(w)=\sum_{k=1}^{\infty}
\ip{f\circ\vp_{z}(w)k_{z}(w)}{e_k}_{\h}e_k.
$$
Each $U_z$ is invertible with the inverse given by 
$$ 
U_z^{-1}f(w):=\frac{1}{k_z(\vf_z(w))}f(\vp_z(w))
=\sum_{k=1}^{\infty}\ip{\frac{1}{k_z(\vf_z(w))}f\circ\vp_{z}(w)}{e_k}_{\h}e_k.
$$
The inverse also satisfies $\norm{U^{-1}_zf}_{\bo}\simeq \norm{f}_{\bo}$. 
Therefore, for every $f\in\bo$ there holds
$$ 
\norm{f}_{\bo}^2=\ip{U_z^*f}{U_z^{-1}f}_{\bo}\leq 
\norm{U^*_zf}_{\bo}\norm{U_z^{-1}f}_{\bo}
\lesssim \norm{U_z^*f}_{\bo}\norm{f}_{\bo}.
$$
This implies that also $\norm{U^*_zf}_{\bo}\simeq \norm{f}_{\bo}$. We will also
use the symbols $U_z$ to denote the operators on $\Bo$ given by the 
formula: 
\begin{align*}
U_zh(w)=h(\vp_z(w))k_z(w)
\end{align*}
for every $h\in\Bo$. It will be clear from context which is meant.

\begin{lm} The following quasi-equalities hold for all
$f\in\bo$ and for all $g\in\Bo$:
\begin{itemize}
\item[(a)] $\abs{U_zg}\simeq \abs{g}$,
\item[(b)] $\abs{U_z^2g}\simeq\abs{g}$, 
\item[(c)] $|U_z^*k_w|\simeq|k_{\vp_z(w)}|$. 
\item[(a')] $\norm{U_zf}_{\bo}\simeq \norm{f}_{\bo}$,
\item[(b')] $\norm{U_z^2f}_{\bo}\simeq\norm{f}_{\bo}$, 
\end{itemize}
\end{lm}

\begin{proof} 
Assertions (a)-(c) were proven in \cite{MW}*{Lemma 2.9}. We use them to 
prove assertions (a') and (b').
Note that 
\begin{align*}
\norm{f}_{\bo}^{2}=\sum_{k}\norm{\ip{f}{e_k}_{\h}}_{\Bo}^{2}
\end{align*}
and 
\begin{align*}
U_{z}f(w)=\sum_{k}\ip{f\circ\vp_{z}(w)k_{z}(w)}{e_k}_{\h}e_k.
\end{align*}
To prove $(a')$, there holds:
\begin{align*}
\norm{U_zf}_{\bo}^2=\norm{\sum_{k=1}^{\infty}
    \ip{f\circ\vp_{z}(w)k_{z}(w)}{e_k}_{\h}e_k}_{\bo}^2
&=\sum_{k=1}^{\infty}
    \norm{\ip{f\circ\vp_{z}(w)k_{z}(w)}{e_k}_{\h}}_{\Bo}^2
\\&\simeq\sum_{k=1}^{\infty}
    \norm{\ip{f}{e_k}_{\h}}_{\Bo}^2
\\&=\norm{f}_{\bo}^{2}.
\end{align*}
Assertion $(b')$ is proven similarly. 
\end{proof}

In case of a strong $\h$--valued Bergman-type space, the $U_z$ are actually 
unitary
operators. 
Moreover, in this case, $U_z^2=I$ and for $u,w,z\in\Om$ and $e\in\h$, there
holds
\begin{align}\label{trans}
U_z(k_we)(u)=
U_z^*(k_we)(u)=\ip{k_w}{k_w}_{\Bo}\norm{K_{\vp_z(w)}}_{\Bo}k_{\vp_z(w)}(u)e.
\end{align}
Since $\abs{\ip{k_w}{k_w}_{\Bo}}\norm{K_{\vp_z(w)}}_{\Bo}=1$, 
this also implies that $\norm{U_z^*k_we}_{\h}=
\norm{k_{\vp_z(w)}e}_{\h}=\norm{U_z k_we}_{\h}$.

For any given operator $T$ on $\bo$ and $z\in\Om$ we define $T^z:=U_zTU^*_z$. 

\subsection{Toeplitz Operators on $\h$--Valued Bergman-type Spaces}
An operator--valued function $u:\Om\to\l(\h)$ will be called 
measurable (analytic) if the function $z\mapsto \ip{u(z)e_k}{e_i}_{\h}$ is
measurable (analytic) for every $i,k\in\N$.
Let $u:\Om\to\l(\h)$ be measurable. Define $M_{u}$ as the operator on $\bo$ 
given by the formula:
$$
(M_{u}f)(z)=u(z)f(z).
$$
Define the Toeplitz operator with symbol $u$ by:
$$
T_u:=PM_u,
$$ 
where $P$ is the usual projection operator onto $\bo$. Let 
$\lh$ be the set of functions
$u:\Om\to \l(\h)$ such that $w\mapsto \norm{u(w)}_{\l(\h)}$ is in 
$L^{\infty}(\Om,\C;d\sigma)$. When $u\in \lh$, it is immediate to see that 
$\norm{T_u}\leq \norm{u}_{\lh}$. In the next section we will 
provide a condition on $u$ which will guarantee that $T_u$ is bounded. 

We are going to further refine this class of Toeplitz operators. We say 
that the function $u\to\mathcal{L}(\h)$ is in $\lf$ if 
$u$ is finite and $u\in\lh$. In other
words, a function $u\in\lf$ may be viewed as a $d\times d$ matrix--valued
function with bounded entries. 
These Toeplitz operators are the key building blocks of an important object for
this paper, the Toeplitz algebra, denoted by $\mathcal{T}_{\lf}$, associated to
the symbols in $\lf$. Specifically, we define
$$
\mathcal{T}_{\lf}:=\textnormal{clos}_{\mathcal{L}(\bo)}
\left\{\sum_{l=1}^L 
\prod_{j=1}^J T_{u_{j,l}}: u_{j,l}\in \lf, J, L \textnormal{ finite}\right\}
$$
where the closure is taken in the operator norm topology on 
$\mathcal{L}(\bo)$.

In the case of strong $\h$--valued Bergman-type spaces, conjugation by 
translations behaves particularly well with respect to Toeplitz operators. 
Namely, if $T=T_u$ is a Toeplitz operator then $T_u^z=T_{u\circ \varphi_{z}}$. 
Moreover, when  $T=T_{u_1}T_{u_2}\cdots T_{u_n}$ is a product of Toeplitz 
operators there holds
$$ T^z=T_{u_1\circ \varphi_{z}}T_{u_2\circ \varphi_{z}}
\cdots T_{u_n\circ \varphi_{z}}.
$$
 
The following lemma is easily deduced from \cite{MW2}*{Lemma 2.10} and
will be used in what follows.

\begin{lm}\label{ToeplitzCompact}
For each bounded Borel set $G$ in $\Om$, and each $d\in\N$, the Toeplitz 
operator $T_{1_G}\ml=\ml T_{1_G}$ is compact on $\bo$.
\end{lm}

\subsection{Geometric Decomposition of 
\texorpdfstring{$(\Om, \m, \la)$}{the Domain}}

The proof of the crucial localization result from Section~\ref{RKTComp} will 
make critical use of the following covering result. For the proof see 
\cite{MW2}. Related results can be found in \cites{CR,Sua,MSW,BI} where it is 
shown that nice domains, such as the unit ball, polydisc, or $\C^n$ have this 
property.

\begin{prop} 
\label{Covering}
There exists an integer $N>0$ (depending only on the doubling constant of the 
measure $\la$) such that for any $r>0$ there is a covering $\FF_r=\{F_j\}$ of 
$\Om$ by disjoint Borel sets satisfying
\begin{enumerate}
\item[\label{Finite} \textnormal{(1)}] every point of $\Om$ belongs to at most 
$N$ of the sets $G_j:=\{z\in\Om: \m(z, F_j)\leq r\}$,
\item[\label{Diameter} \textnormal{(2)}] $\textnormal{diam}_{\m}\, F_j \leq 4r$ 
for every $j$.
\end{enumerate}
\end{prop}

\section{Reproducing kernel thesis for boundedness}\label{bdd}
In this section, we will give sufficient conditions for boundedness of operators
on $\bo$. Ideally, we would like to show that the conditions:
\begin{align*}
\sup_{k}\sup_{z\in\Om}\norm{U_zTk_ze_k}_{L^{p}(\Om,\ell^p;d\sigma)}^{p}
=\sup_{k}\sup_{z\in\Om}\sum_{i=1}^{\infty}
\norm{\ip{U_zT(k_ze_k)}{e_i}}_{L^{p}(\Om,\C;d\sigma)}^{p}<\infty
\end{align*}
and
\begin{align*}
\sup_{i}\sup_{z\in\Om}\norm{U_zT^*k_ze_i}_{L^{p}(\Om,\ell^p;d\sigma)}^{p}
=\sup_{i}\sup_{z\in\Om}\sum_{k=1}^{\infty}
\norm{\ip{U_zT(k_ze_i)}{e_k}}_{L^{p}(\Om,\C;d\sigma)}^{p}<\infty,
\end{align*}
are enough to guarantee that $T$ is bounded. However, if $T$ satisfies a
stronger condition, we can conclude that $T$ is bounded.

\begin{thm}\label{RKT} Let $T:\bo\to\bo$ be a linear operator defined a priori 
only on the linear span of normalized reproducing kernels of $\bo$. Assume 
that there exists an operator $T^*$ defined on the same span such that the 
duality relation $\ip{Tk_ze}{k_wh}_{\bo}=\ip{k_ze}{T^*k_wh}_{\bo}$ 
holds for all 
$z,w\in\Om$ and all finite $e,h\in\h$. Let  $\kappa$ be the constant from 
\hyperref[A6]{A.6}. 
If
\begin{align}\label{e11}
\sup_{i}\sup_{z\in\Om}\left\{\int_{\Om}\left(\sum_{k=1}^{\infty}
    \abs{\ip{U_zT^*(k_ze_i)(u)}
    {e_k}_{\h}}\right)^{p}d\sigma(u)\right\}^{\frac{1}{p}}<\infty,
\end{align}
and
\begin{align}\label{e21}
\sup_{k}\sup_{z\in\Om}\left\{\int_{\Om}\left(\sum_{i=1}^{\infty}
    \abs{\ip{U_zT(k_ze_k)(u)}
    {e_i}_{\h}}\right)^{p}d\sigma(u)\right\}^{\frac{1}{p}}<\infty
\end{align}
for some $p>\frac{4-\kappa}{2-\kappa}$ then $T$ can be 
extended to a bounded operator on $\bo$.  
\end{thm}
\begin{rem}
Note that by Minkowski's inequality, the above conditions can be replaced
by
\begin{align}\label{e12}
\sup_{i}\sup_{z\in\Om}\sum_{k=1}^{\infty}
    \norm{\ip{U_zT^*(k_ze_i)}{e_k}_{\h}}_{L^{p}(\Om,\C;d\sigma)}<\infty
\end{align}
and
\begin{align}\label{e22}
\sup_{k}\sup_{z\in\Om}\sum_{k=1}^{\infty}
    \norm{\ip{U_zT(k_ze_k)}{e_i}_{\h}}_{L^{p}(\Om,\C;d\sigma)}<\infty.
\end{align}
We state the theorem with conditions \eqref{e11} and \eqref{e21} since they are,
in general, smaller than the quantities in \eqref{e12} and \eqref{e22}. 
Similar statments are true for all of the theorems in this section.
\end{rem}
\begin{proof}  
Since the linear span of the normalized reproducing kernels is dense in 
$\bo$ it will be enough to show that there exists a finite constant such 
that $\norm{Tf}_{\bo}\lesssim\norm{f}_{\bo}$ for all $f$ that are in the linear 
span of the normalized reproducing kernels.
Notice first that for any such $f$ there holds
\begin{align}\label{t_est}
\int_{\Om}\norm{(Tf)(z)}_{\h}^{2}d\sigma(z) \notag
&=\int_{\Om}\norm{\sum_{i=1}^{\infty}\notag
    \ip{Tf}{K_ze_i}_{\bo}e_i}_{\h}^{2}d\sigma(z)
\\&=\int_{\Om}\norm{\sum_{i=1}^{\infty}\notag
    \int_{\Om}\sum_{k=1}^{\infty}
    \ip{f_k(w)e_k}{T^*K_ze_i}_{\h}e_id\sigma(w)}_{\h}^{2}d\sigma(z)
\\&=\int_{\Om}\norm{\sum_{i=1}^{\infty}
    \int_{\Om}\sum_{k=1}^{\infty}
    f_k(w)\ip{K_we_k}{T^*K_ze_i}_{\bo}e_id\sigma(w)}_{\h}^{2}d\sigma(z)\notag
\\&\leq\int_{\Om}\norm{
    \int_{\Om}\sum_{i=1}^{\infty}\sum_{k=1}^{\infty}
    \abs{f_k(w)}\abs{
    \ip{K_we_k}{T^*K_ze_i}_{\bo}}e_id\sigma(w)}_{\h}^{2}d\sigma(z)
\\&=\int_{\Om}\norm{\int_{\Om}M(z,w)\abs{f(w)}d\sigma(w)}_{\h}^{2}d\sigma(z)\notag.
\end{align}
In \eqref{t_est}, we use the fact that 
\begin{align*}
\norm{\sum_{i=1}^{\infty}\lambda_ie_i}_{\h}^{2} = 
\norm{\sum_{i=1}^{\infty}\abs{\lambda_i}e_i}_{\h}^{2}
\end{align*}
and we define
$$
\abs{f(w)}:=\sum_{k=1}^{\infty}\abs{\ip{f_k(w)}{e_k}_{\h}}e_k.
$$ 
Thus,
we only need to show that the integral operator with matrix--valued kernel
$M(z,w)$ is bounded from $L^{2}(\Om,\h;d\sigma)\to L^{2}(\Om,\h;d\sigma)$, where
\begin{align*}
\ip{M(z,w)e_k}{e_i}_{\h}=\abs{\ip{K_we_k}{T^*(K_ze_i)}_{\bo}}.
\end{align*}
The Matrix Schur's Test, (Lemma \ref{MSchur}), will be used to prove that this 
operator is bounded. We set
\[\ip{M(z,w)e_k}{e_i}_{\h}=\abs{\ip{K_we_k}{T^*(K_ze_i)}_{\bo}},\ 
\hspace{0.15cm} h(z) \equiv \norm{\Kbz}_{\Bo}^{\al/2}, \]
\[
\hspace{0.15cm} 
X = \Om, 
\hspace{0.15cm} d\mu(z) =  
d\nu(z) = d\sigma(z).\] If $\kappa=0$ set $\al=\frac{4-2\kappa}{4-\kappa}=1$. 
If $\kappa>0$ choose $\alpha\in (\frac{2}{p}, \frac{4-2\kappa}{4-\kappa})$ 
such that $q(\alpha-\frac{2}{p})<\kappa$. The condition $p>\frac{4-\kappa}
{2-\kappa}$ ensures that such $\alpha$ exists. Let $z\in\Om$ be arbitrary and 
fixed. There holds
\begin{align*}
Q_1
:&=\int_{\Om}\sum_{k=1}^{\infty}\norm{K_w}_{\Bo}^{\al}
    \ip{M(z,w)e_k}{e_i}_{\h}d\sigma(w)
\\&=\int_{\Om}\sum_{k=1}^{\infty}\abs{\ip{K_we_k}{T^*(K_ze_i)}_{\bo}}
    \norm{K_w}_{\Bo}^{\al}d\sigma(w)
\\&=\norm{K_z}_{\Bo}\int_{\Om}\sum_{k=1}^{\infty}
    \abs{\ip{T^*(k_ze_i)}{k_{\vp_z(u)}}_{\bo}}
    \norm{K_{\vp_z(u)}}_{\Bo}^{\al-1}d\lambda(w)
\\&\simeq\norm{K_z}_{\Bo}\int_{\Om}\sum_{k=1}^{\infty}
    \abs{\ip{T^*(k_ze_i)}{U_z^*k_ue_k}_{\bo}}
    \abs{\ip{k_z}{k_u}_{\Bo}}^{1-\al}d\lambda(w)
\\&=\norm{K_z}_{\Bo}\int_{\Om}\sum_{k=1}^{\infty}
    \abs{\ip{U_zT^*(k_ze_i)}{k_ue_k}_{\bo}}
    \abs{\ip{k_z}{k_u}_{\Bo}}^{1-\al}d\lambda(w)
\\&=\norm{K_z}_{\Bo}^{\al}\int_{\Om}
    \sum_{k=1}^{\infty}\abs{\ip{U_zT^*(k_ze_i)(u)}{e_k}_{\h}}
    \frac{\abs{\ip{K_z}{K_u}_{\Bo}}^{2-\al}}
    {\norm{K_u}_{\Bo}^{2-\al}}d\lambda(u).
\end{align*}
By H\"{o}lder's Inequality, this quantity is no worse than:
\begin{align*}
\norm{K_z}_{\Bo}^{\al}
    \left\{\int_{\Om}\left(\sum_{k=1}^{\infty}
    \abs{\ip{U_zT^*(k_ze_i)(u)}
    {e_k}_{\h}}\right)^{p}d\sigma(u)\right\}^{\frac{1}{p}}
    \left\{\int_{\Om}\frac{\abs{\ip{K_z}{K_u}_{\Bo}}^{q(1-\al)}}
    {\norm{K_u}_{\Bo}^{q(2-\al-\frac{2}{p})}}d\lambda(u)\right\}^{\frac{1}{q}}.
\end{align*}
Let $r=q\left(2-\al-\frac{2}{p}\right)$ and $s=r-2q(1-\al)$. Then 
$r=\frac{p(2-\al-\frac{2}{p})}{p-1}=2-\frac{\al p}{p-1} > \kappa$ and
$s=q(\al-\frac{2}{p})<\kappa$ when $\kappa>0$ and $s=r>\kappa$ if
$\kappa=0$. This means that both $r$ and $s$ satisfy all condition of 
\hyperref[A6]{A.6}. Thus, by Lemma~\ref{RF}, the second integral is bounded 
independent of $z$. Call this constant $C$. This gives that:
\begin{align*}
Q_1
&\leq C\norm{K_z}_{\Bo}^{\al}
    \sup_{i}\sup_{z\in\Om}\left\{\int_{\Om}\left(\sum_{k=1}^{\infty}
    \abs{\ip{U_zT^*(k_ze_i)(u)}
    {e_k}_{\h}}\right)^{p}d\sigma(u)\right\}^{\frac{1}{p}}.
\end{align*}
By interchanging the roles of $T$ and $T^*$ and $i$ and $k$, we similarly
obtain:
\begin{align*}
Q_2
:&=\int_{\Om}\sum_{i=1}^{\infty}\norm{K_z}_{\Bo}^{\al}
    \ip{M^*(z,w)e_i}{e_k}_{\h}d\sigma(z)
\\&=\int_{\Om}\sum_{i=1}^{\infty}\norm{K_z}_{\Bo}^{\al}
    \abs{\ip{K_we_k}{T^*(K_ze_i)}_{\bo}}d\sigma(z)
\\&=\int_{\Om}\sum_{i=1}^{\infty}\norm{K_z}_{\Bo}^{\al}
    \abs{\ip{T(K_we_k)}{K_ze_i}_{\bo}}d\sigma(z)
\\&\leq C\norm{K_w}_{\Bo}^{\al}
    \sup_{e_k}\sup_{z\in\Om}\left\{\int_{\Om}\left(\sum_{i=1}^{\infty}
    \abs{\ip{U_zT(k_ze_k)(u)}
    {e_i}_{\h}}\right)^{p}d\sigma(u)\right\}^{\frac{1}{p}}.
\end{align*}
Thus, by the Matrix Schur's Test (Lemma~\ref{MSchur}) and our assumptions, the
operator is bounded. 
\end{proof}

\subsection{RKT for Toeplitz operators}
In the case when $T=T_F$ is a Toeplitz operator, the conditions in 
Theorem \ref{RKT} can be stated in terms of the symbol, $F$. 

\begin{cor} Let $\bo$ be a strong Bergman-type space.
If $\kappa>0$ and $T_F$ is a Toeplitz operator whose symbol $F$ satisfies 
\begin{align*}
\sup_{i}\sup_{z\in\Om}\sum_{k=1}^{\infty}
    \norm{\ip{(F^*\circ\vp_z)e_i}{e_k}_{\h}}_{L^{p}(\Om,\C;d\sigma)}<
    \infty
\end{align*}
and
\begin{align*}
\sup_{k}\sup_{z\in\Om}\sum_{i=1}^{\infty}
    \norm{\ip{(F\circ\vp_z)e_k}{e_i}_{\h}}_{L^{p}(\Om,\C;d\sigma)}
    <\infty,
\end{align*}
for some $p>\frac{4-\kappa}{2-\kappa}$ then $T_F$ is bounded on $\bo$.
\end{cor}

\begin{proof} We first show that for all finite $e\in\h$ there holds
\begin{align*}
\abs{\ip{(U_zT^*_Fk_ze_i)(w)}{e_k}_{\h}}
=\abs{P\left(\ip{(F^*\circ\vp_z)e_i}{e_k}_{\h}\right)(w)}
\end{align*}
and
\begin{align*}
\abs{\ip{(U_zT_Fk_ze_k)(w)}{e_i}_{\h}}=
\abs{P\left(\ip{(F\circ\vp_z)e_k}{e_i}_{\h}\right)(w)}.
\end{align*}

By~\hyperref[A5]{A.5}, $\abs{k_0}\equiv 1$ on $\Om$. By the maximum and 
minimum modulus
principles, this means that $k_0$ is constant on $\Om$ and since 
$k_0(0)=\norm{K_0}_{\Bo}>0$ there holds that $k_0\equiv 1$ on $\Om$.
Equation \eqref{trans} will be used several times.
\begin{align*}
\abs{\ip{(U_zT_Fk_ze_k)(w)}{e_i}_{\h}}
&=\norm{K_w}_{\Bo}\abs{\ip{U_zT_Fk_ze_k}{k_we_i}_{\bo}}
\\&=\norm{K_w}_{\Bo}\abs{\int_{\Om}
    \ip{F(a)k_z(a)e_k}{k_{\vp_z(w)}(a)e_i}_{\h}d\sigma(a)}
\\&=\norm{K_w}_{\Bo}\abs{\int_{\Om}
    \ip{F(a)e_k}{e_i}_{\h}\ip{k_z}{k_a}_{\Bo}
    \overline{\ip{k_{\vp_z(w)}}{k_a}_{\Bo}}d\lambda(a)}
\\&=\norm{K_w}_{\Bo}\abs{\int_{\Om}
    \ip{F(\vp_z(b))e_k}{e_i}_{\h}\ip{k_z}{k_{\vp_z(b)}}_{\Bo}
    \overline{\ip{k_{\vp_z(w)}}{k_{\vp_z(b)}}_{\Bo}}d\lambda(b)}
\\&=\abs{\int_{\Om}
    \ip{F(\vp_z(b))e_k}{e_i}_{\h}\
    k_0(b)
    \overline{\ip{K_w}{K_b}_{\Bo}}d\lambda(b)}
\\&=\abs{P\left(\ip{(F\circ\vp_z)e_k)}{e_i}_{\h}\right)(w)}.
\end{align*}
And $\abs{\ip{(U_zT^*_Fk_ze_i)(w)}{e_k}_{\h}}
=\abs{P\left(\ip{(F^*\circ\vp_z)e_i}{e_k}_{\h}\right(w)}$ is proven similarly. 

Therefore, by the boundedness of the (scalar--valued) Bergman projection, 
Lemma~\ref{projs}, there 
holds:
\begin{align*}
\sup_{i}\sup_{z\in\Om}\left\{\int_{\Om}\left(\sum_{k=1}^{\infty}
    \abs{\ip{U_zT_F^*(k_ze_i)(u)}
    {k}_{\h}}\right)^{p}d\sigma(u)\right\}^{\frac{1}{p}}
&\leq \sup_{e_i}\sup_{z\in\Om}\sum_{k=1}^{\infty}
    \norm{\ip{(F^*\circ\vp_z)e_i}{e_k}_{\h}}_{L^{p}(\Om,\C;d\sigma)}
\end{align*}
and
\begin{align*}
\sup_{e_k}\sup_{z\in\Om}\left\{\int_{\Om}\left(\sum_{i=1}^{\infty}
    \abs{\ip{U_zT_F(k_ze_k)(u)}
    {i}_{\h}}\right)^{p}d\sigma(u)\right\}^{\frac{1}{p}}
\leq \sup_{k}\sup_{z\in\Om}\sum_{i=1}^{\infty}
    \norm{\ip{(F\circ\vp_z)e_k}{e_i}_{\h}}_{L^{p}(\Om,\C;d\sigma)}.
\end{align*}
Therefore, the two conditions from Theorem \ref{RKT} are satisfied and so
$T_F$ is bounded.
\end{proof}

\subsection{RKT for product of Toeplitz operators with analytic symbols}
In
this section we derive a sufficient condition for boundedness of products 
Toeplitz operators, $T_{F}T_{\adj{G}}$. For another result giving 
sufficient conditions for the boundedness of this product see \cite{K}.
\begin{cor}\label{cor:prodan}
Let $\bo$ be a strong $\h$--valued Bergman-type space such that a 
product of any
two reproducing kernels from $\Bo$ is still in $\Bo$. Let 
$F,G:\Om\to\l(\h)$ satisfy $\ip{Fe_k}{e_i}_{\h},\ip{Ge_k}{e_i}_{\h}\in\bo$ for 
every $i,k\in\N$. 
If there exists $p>\frac{4-\kappa}{2-\kappa}$ such that
\begin{align*}
\sup_{k}\sup_{z\in\Om}\sum_{i=1}^{\infty}\norm{\ip{\adj{G}(z)e_k}
    {\adj{F}\circ\vp_ze_i}_{\h}}_{L^{p}(\Om,\C,d\sigma)}<\infty
\end{align*}
and
\begin{align*}
\sup_{i}\sup_{z\in\Om}\sum_{k=1}^{\infty}\norm{\ip{\adj{F}(z)e_i}
    {\adj{G}\circ\vp_ze_k}_{\h}}_{L^{p}(\Om,\C,d\sigma)}<\infty
\end{align*}
then the operator $T_FT_{\adj{G}}$ is bounded on $\bo$.
\end{cor}

\begin{proof} We only need to check that $T_FT_{\adj{G}}$ satisfies the 
conditions of Theorem~\ref{RKT}. 
We first show that
$\ip{T_{\adj{G}}k_{z}e_i}{e_k}_{\h}=\ip{\adj{G(z)}k_z(w)e_i}{e_k}_{\h}$. 
First assume
that $\ip{\adj{G}e_i}{e_k}_{\h}$ is a finite linear combination of reproducing 
kernels. Then $K_w\ip{\adj{G}e_i}{e_k}_{\h}=\ip{K_w\adj{G}e_i}{e_k}_{\h}
\in\Bo$ for any reproducing kernel $K_w$. Therefore, 
\begin{align*}
\ip{T_{\adj{G}}k_z(w)e_i}{e_k}_{\h}
&=\int_{\Om}\ip{\ip{K_u}{K_w}_{\Bo}\adj{G}(u)k_z(u)e_i}{e_k}_{\h}d\sigma(u)
\\&=\overline{\ip{K_wGe_k}{k_ze_I}_{\bo}}
\\&=\norm{K_z}_{\Bo}^{-1}\overline{\ip{K_wGe_k}{K_ze_i}_{\bo}}
\\&=\ip{\adj{G(z)}k_z(w)e_i}{e_k}_{\h}.
\end{align*}
Next, let $G$ be arbitrary. Fix $z, w\in\Om$. 
Let $\epsilon>0$. There is a matrix--valued $H:\Om\to\l(\h)$ such that 
$\ip{He_i}{e_k}_{\h}$ is a finite linear combination of reproducing kernels
and $\norm{\ip{(G-H)e_i}{e_k}_{\h}}_{\Bo}<\epsilon$ and
$\norm{\ip{(G-H)e_k}{e_i}_{\h}}_{\Bo}<\epsilon$. That is, $H$ is a 
matrix--valued function and the entries of $H$ approximate the entries 
of $G$. Note that we are not claiming that $H$ converges to $G$ in any 
operator norm, this is only convergence in $\Bo$ of the entries of $H$ to the
entries of $G$. Then there holds
\begin{align*} 
|\ip{T_{\adj{G}}k_z(w)e_i}{e_k}_{\h}
    &-\ip{\adj{H}(z)k_z(w)e_i}{e_k}_{\h}|
    =\abs{\ip{T_{\adj{(G-H)}}k_z(w)e_i}{e_k}_{\h}} 
\\&=\abs{\int_{\Om} 
    \ip{\ip{K_u}{K_w}_{\Bo}(\adj{(G(u)-H(u))})k_z(u)e_i}{e_k}_{\h}d\sigma(u)}
\\&=\abs{\int_{\Om}\overline{K_w(u)}k_z(u)
    \ip{(\adj{(G(u)-H(u))})e_i}{e_k}_{\h}d\sigma(u)}
\\&\leq \int_{\Om} \abs{K_w(u)k_z(u)}^2d\sigma(u)
    \norm{\ip{(G-H)e_k}{e_i}_{\h}}_{\Bo}^2 
\\&<C(z, w)\epsilon^2.
\end{align*}
Moreover, 
\begin{align*}
\abs{\ip{(G(z)-H(z))e_i}{e_k}_{\h}}
&=\abs{\ip{\ip{(G(z)-H(z))e_i}{e_k}_{\h}}{K_{z}}_{\Bo}}
\\&\leq\norm{K_z}_{\Bo}\norm{\ip{(G-H)e_i}{e_k}_{\h}}_{\Bo}
\\&<\norm{K_z}_{\Bo}\epsilon.
\end{align*} 
Since $\epsilon>0$ was arbitrary and $z, w$ were fixed there holds
$\ip{T_{\adj{G}}k_{z}e_i}{e_k}_{\h}=\ip{\adj{G(z)}k_z(w)e_i}{e_k}_{\h}$ and 
$\ip{T_{\adj{F}}k_{z}e_k}{e_i}_{\h}=\ip{\adj{F(z)}k_z(w)e_k}{e_i}_{\h}$. It is 
also easy to see that this implies 
$\ip{T_{\adj{G}}k_{z}e_i}{fe_k}_{\bo}=\ip{\adj{G(z)}k_ze_i}{fe_k}_{\bo}$ for
$f\in\Bo$.

So, there holds:
\begin{align*}
\abs{\ip{U_zT_FT_{\adj{G}}k_ze_k(w)}{e_i}_{\h}}
&=\norm{K_w}_{\Bo}\abs{\ip{U_zT_FT_\adj{G}k_ze_k}{k_we_i}_{\bo}}
\\&=\norm{K_w}_{\Bo}\abs{\ip{T_FT^*_{G}(z)k_ze_k}{k_{\vp_z(w)}e_i}_{\bo}}
\\&=\norm{K_w}_{\Bo}\abs{\ip{\adj{G}(z)k_ze_k}
    {\adj{F}\circ\vp_z(w)k_{\vp_z(w)}e_i}_{\bo}}
\\&=\norm{K_w}_{\Bo}\abs{
    \ip{\adj{G}(z)e_k}{\adj{F}\circ\vp_z(w)e_i}_{\h}
    \ip{k_z}{k_{\vp_z(w)}}_{\Bo}}
\\&=\abs{\ip{\adj{G}(z)e_k}{\adj{F}\circ\vp_z(w)e_i}_{\h}}.
\end{align*}
Thus, 
\begin{align*}
\abs{\ip{U_zT_FT_{\adj{G}}k_ze_k(w)}{e_i}_{\h}}=
\abs{\ip{\adj{G}(z)e_k}{\adj{F}\circ\vp_z(w)e_i}_{\h}}.
\end{align*}
and
\begin{align*}
\abs{\ip{U_zT_GT_{\adj{F}}k_ze_i(w)}{e_k}_{\h}}=
\abs{\ip{\adj{F}(z)e_i}{\adj{G}\circ\vp_z(w)e_k}_{\h}}.
\end{align*}
Using our hypotheses, we deduce that $T_FT_{\adj{G}}$ satisfies the conditions
of Theorem \ref{RKT}.
\end{proof}

\subsection{RKT for Hankel operators} Next we treat the case of Hankel 
operators.  
The Hankel operator $H_F:\bo\to \bo^{\perp}$ with matrix--valued symbol 
$F:\Om\to \l(\h)$ is 
defined by $H_Fg=(I-P)Fg$, where $P$ is the orthogonal projection of 
$L^2(\Om,\h;d\sigma)$ onto $\bo$. Since $H_F$ is not a operator from 
$\bo$ to $\bo$, we can't apply Theorem~\ref{RKT}. However, we can reuse 
the proof to prove the following Corollary.  
 
\begin{cor} Let $\bo$ be a strong Bergman-type space.
If $H_{F}$ is a Hankel operator whose symbol $F$ satisfies 
\begin{align*}
\sup_{i}\sup_{z\in\Om}
    \left(\int_{\Om}\left(\sum_{k=1}^{\infty}
    \abs{\ip{(F(z)-F(\vp_z(u)))e_k}{e_i}_{\h}}
    \right)^pd\sigma(u)\right)^{\frac{1}{p}}<\infty
\end{align*}
for some $p>\frac{4-\kappa}{2-\kappa}$ then $H_{F}$ is bounded.
\end{cor} 

\begin{proof} The proof is basically the same as for Theorem~\ref{RKT}. 
As in the proof the Theorem~\ref{RKT}, we show that there is a constant such
that 
\begin{align*}
\norm{H_Fg}_{\bo}\lesssim\norm{g}_{\bo}
\end{align*}
for any $g\in\Bo$ that is a linear combination of normalized reproducing
kernels. First, there holds:
\begin{align*}
(H_{F}g)(z)&=F(z)g(z)-P(Fg)(z)
\\&=\int_{\Om}\left(F(z)g(w)-F(w)f(w)\right)\ip{K_w}{K_z}_{\Bo}d\sigma(w)
\\&=\int_{\Om}\sum_{i=1}^{\infty}\sum_{k=1}^{\infty}
    \ip{(F(z)-F(w))e_k}{e_i}_{\h}\ip{K_w}{K_z}_{\Bo}g_k(w)e_id\sigma(w)
\end{align*}
Thus, we want to show that the integral operator with matrix--valued kernel
given by:
\begin{align*}
\ip{M(z,w)e_k}{e_i}_{\h}=\abs{\ip{(F(z)-F(w))e_k}{e_i}_{\h}}
    \abs{\ip{K_z}{K_w}_{\Bo}}
\end{align*}
is bounded. The Matrix Schur's Test, (Lemma \ref{MSchur}), will be used to prove
that the 
operator is bounded with 
\[\ip{M(z,w)e_k}{e_i}_{\h}=\abs{\ip{(F(z)-F(w))e_k}{e_i}_{\h}}
    \abs{\ip{K_z}{K_w}_{\Bo}}, 
\hspace{0.15cm} h(z) \equiv \norm{\Kbz}_{\Bo}^{\al/2}, \]
\[X = \Om, 
\hspace{0.15cm} d\mu(z) =  
d\nu(z) = d\sigma(z).\] 
If $\kappa=0$ set $\al=\frac{4-2\kappa}{4-\kappa}=1$. 
If $\kappa>0$ choose $\alpha\in (\frac{2}{p}, \frac{4-2\kappa}{4-\kappa})$ 
such that $q(\alpha-\frac{2}{p})<\kappa$. The condition $p>\frac{4-\kappa}
{2-\kappa}$ ensures that such $\alpha$ exists. Let $z\in\Om$ be arbitrary and 
fixed. There holds:
\begin{align*}
Q_1
&:=\int_{\Om}\sum_{k=1}^{\infty}\norm{K_w}_{\Bo}^{\al}
    \ip{M(z,w)e_k}{e_i}_{\h}d\sigma(w)
\\&=\int_{\Om}\sum_{k=1}^{\infty}\norm{K_w}_{\Bo}^{\al}
    \abs{\ip{(F(z)-F(w))e_k}{e_i}_{\h}}
    \abs{\ip{K_z}{K_w}_{\Bo}}d\sigma(w)
\\&=\norm{K_z}_{\Bo}\int_{\Om}\sum_{k=1}^{\infty}
    \abs{\ip{(F(z)-F\circ\vp_z(u))e_k}{e_i}_{\h}}
    \abs{\ip{k_z}{k_{\vp_z(u)}}_{\Bo}}
    \norm{K_{\vp_z(u)}}_{\Bo}^{\al-1}d\lambda(u)
\\&\simeq\norm{K_z}_{\Bo}\int_{\Om}\sum_{k=1}^{\infty}
    \abs{\ip{(F(z)-F\circ\vp_z(u))e_k}{e_i}_{\h}}
    \abs{\ip{k_z}{U_z^{*}k_u}_{\Bo}}
    \abs{\ip{k_z}{k_u}_{\Bo}}^{1-\al}d\lambda(u)
\\&=\norm{K_z}_{\Bo}\int_{\Om}\sum_{k=1}^{\infty}
    \abs{\ip{(F(z)-F\circ\vp_z(u))e_k}{e_i}_{\h}}
    \abs{\ip{U_zk_z}{k_u}_{\Bo}}
    \abs{\ip{k_z}{k_u}_{\Bo}}^{1-\al}d\lambda(u)
\\&=\norm{K_z}_{\Bo}^{\al}\int_{\Om}\sum_{k=1}^{\infty}
    \abs{\ip{(F(z)-F\circ\vp_z(u))e_k}{e_i}_{\h}}
    \frac{\abs{\ip{K_z}{K_u}_{\Bo}}^{1-\al}}
    {\norm{K_u}_{\Bo}^{2-\al}}d\lambda(u).
\end{align*}
Using Holder's inequality we obtain that the last expression is no greater than
$$
\norm{\Kbz}_{\Bo}^{\al}
\left(\int_{\Om}\left(\sum_{k=1}^{\infty}
\abs{\ip{(F(z)-F(\vp_z(u)))e_k}{e_i}_{\h}}
\right)^pd\sigma(u)\right)^{\frac{1}{p}}
\left(\int_{\Om}\frac{\abs{\ip{\Kbz}
{K_{u}}}^{q(1-\al)}}{\norm{K_u}^{q\left(2-\al-\frac{2}{p}\right)}}\,
d\la(u)\right)^{\frac{1}{q}}.
$$
Let $r=q\left(2-\al-\frac{2}{p}\right)$ and $s=r-2q(1-\al)$. Then 
$r=\frac{p(2-\al-\frac{2}{p})}{p-1}=2-\frac{\al p}{p-1} > \kappa$ and
$s=q(\al-\frac{2}{p})<\kappa$ when $\kappa>0$ and $s=r>\kappa$ if
$\kappa=0$. This means that both $r$ and $s$ satisfy all condition of 
\hyperref[A6]{A.6}. Thus, by Lemma~\ref{RF}, the second integral is bounded 
independet of $z$. Call this constant $C$. This gives that:
\begin{align*}
Q_1\leq\norm{K_z}_{\Bo}C\sup_{i}\sup_{z\in\Om}
    \left(\int_{\Om}\left(\sum_{k=1}^{\infty}
    \abs{\ip{(F(z)-F(\vp_z(u)))e_k}{e_i}_{\h}}
    \right)^pd\sigma(u)\right)^{\frac{1}{p}}.
\end{align*}
Now we check the second condition in Lemma~\ref{MSchur}.
\begin{align*}
Q_2
&:=\int_{\Om}\sum_{i=1}^{\infty}\norm{K_z}_{\Bo}^{\al}
    \ip{M(z,w)^*e_i}{e_k}_{\h}d\sigma(z)
\\&=\int_{\Om}\sum_{i=1}^{\infty}\norm{K_z}_{\Bo}^{\al}
    \ip{e_i}{M(z,w)e_k}_{\h}d\sigma(z)
\\&=\int_{\Om}\sum_{i=1}^{\infty}\norm{K_z}_{\Bo}^{\al}
    \abs{\ip{(F(z)-F(w))e_k}{e_i}_{\h}}
    \abs{\ip{K_z}{K_w}_{\Bo}}d\sigma(w)
\\&=\int_{\Om}\sum_{i=1}^{\infty}\norm{K_z}_{\Bo}^{\al}
    \abs{\ip{e_k}{(F(z)-F(w)^*e_i}_{\h}}
    \abs{\ip{K_z}{K_w}_{\Bo}}d\sigma(w).
\end{align*}
Using similar arguments as above, we conculde that:
\begin{align*}
Q_2&\leq\norm{K_w}_{\Bo}C\sup_{k}\sup_{z\in\Om}
    \left(\int_{\Om}\left(\sum_{i=1}^{\infty}
    \abs{\ip{(F(z)-F(\vp_z(u)))^*e_i}{e_k}_{\h}}
    \right)^pd\sigma(u)\right)^{\frac{1}{p}}
\\&=\norm{K_w}_{\Bo}C\sup_{k}\sup_{z\in\Om}
    \left(\int_{\Om}\left(\sum_{i=1}^{\infty}
    \abs{\ip{(F(z)-F(\vp_z(u)))e_k}{e_i}_{\h}}
    \right)^pd\sigma(u)\right)^{\frac{1}{p}}
\\&=\norm{K_w}_{\Bo}C\sup_{i}\sup_{z\in\Om}
    \left(\int_{\Om}\left(\sum_{k=1}^{\infty}
    \abs{\ip{(F(z)-F(\vp_z(u)))e_k}{e_i}_{\h}}
    \right)^pd\sigma(u)\right)^{\frac{1}{p}}.
\end{align*}
Thus, by our assumptions and the Matrix Schur's Test (Lemma~\ref{MSchur}),
the operator is bounded.
\end{proof}
 
\section{Reproducing kernel thesis for compactness}\label{RKTComp}

Compact operators on a Hilbert space are exactly the ones which send weakly 
convergent sequences into strongly convergent ones. In the current setting,
there are, in essence, two ``layers'' of compactness that must be satisfied. 
For example, let $\vp$ be a scalar--valued function. Then the Toeplitz operator
$T_{\vp I}$ is not compact on $\bo$ (unless $\vp\equiv 0$) since
the sequence $\{T_{\vp I}e_k\}_{k=1}^{\infty}$ 
does not converge strongly to zero in
$\bo$ but the sequence $\{e_k\}_{k=1}^{\infty}$ 
converges weakly to $0$ in $\bo$. 
On the other hand, if $T_{\vp}$ is compact on $\Bo$, then $T_{\vp 
I^{(d)}}$ is compact for any $d\in\N$. 

The goal of this section is to prove that if $T$ 
satisfies the conditions of Theorem~\ref{RKT} and another condition 
to be stated, and if $T$ sends the weakly null sequence 
$\{k_ze\}_{z\in\Om}$ (see Lemma~\ref{compact} below) into 
a strongly null sequence $\{Tk_ze\}_{z\in\Om}$, then $T$ must be compact.

Recall that the essential norm of a bounded linear operator $S$ on $\bo$ is 
given by 
$$
\norm{S}_e=\inf\left\{\norm{S-A}_{\l(\bo)}: 
A\textnormal{ is compact on } \bo\right\}.
$$
We first show two simple results that will be used in the course of the proofs.

\begin{lm}\label{compact} The weak limit of $k_ze$ is zero as 
$\m(z,0)\to\infty$. 
\end{lm}

\begin{proof} Note first that property \hyperref[A2]{A.2} implies that if 
$\m(z,0)\to\infty$ then $\m(\vp_w(z),0)\to \infty$. Properties 
\hyperref[A5]{A.5} and \hyperref[A7]{A.7} now immediately imply that 
$\ip{k_we}{k_zh}_{\bo}\to 0$ as $\m(z,0)\to\infty$. The fact that the set
$\{k_ze_i:z\in\Om, i\in\N\}$ is dense in $\bo$ then implies $k_ze$ converges 
weakly to $0$ as $\m(z,0)\to\infty$. 
\end{proof}

\begin{lm} 
\label{lm:Compactsaregood}
For any compact operator $A$ and any $f\in\bo$ we have that 
$\norm{A^zf}_{\bo}\to 0$ as $\m(z,0)\to\infty$.
\end{lm}

\begin{proof} If $e\in\h$ and $f=k_we$ then using the 
previous lemma we obtain that 
$\norm{A^zk_we}_{\bo}\simeq 
\norm{U_zAk_{\vp_z(w)}e}_{\bo}\to 0$ as $\m(z,0)\to\infty$. 
For the
general case, choose $f\in\bo$ arbitrary of norm $1$.  We can approximate $f$ by
linear combinations of normalized reproducing kernels
and in a standard way we can deduce the same result. 
\end{proof}

The following localization property will be a crucial step towards estimating 
the essential norm. A version of this result in the classical Bergman space 
setting was first proved by Su\'arez in~\cite{Sua}. Related results were later 
given in~\cites{MW,MSW,BI,RW}.

\begin{prop}\label{MainEst1} Let $T:\bo\to\bo$ be a linear operator and $\kappa$
be the constant from \hyperref[A6]{A.6}.
If
\begin{align*} 
\sup_{i}\sup_{z\in\Om}
\lp\int_{\Om}
    \lp\sum_{k=1}^{\infty}\abs{\ip{U_zT^*k_ze_i(u)}{e_k}_{\h}}\rp^{p}
    d\sigma(u)\rp^{\frac{1}{p}}
<\infty
\end{align*}
and
\begin{align*} 
\sup_{k}\sup_{z\in\Om}
\lp\int_{\Om}
    \lp\sum_{i=1}^{\infty}\abs{\ip{U_zTk_ze_k(u)}{e_i}_{\h}}\rp^{p}
    d\sigma(u)\rp^{\frac{1}{p}}
<\infty
\end{align*}
for some $p>\frac{4-\kappa}{2-\kappa}$, then for every 
$\epsilon > 0$ there exists $r>0$ such that for the covering $\FF_r=\{F_j\}_{j=1}^{\infty}$ 
(associated to $r$) from Proposition~\ref{Covering}  
\begin{eqnarray*} \norm{ T-\sum_{j}M_{1_{F_j} }
TPM_{1_{G_j} }}_{\l(\Om,\h;d\sigma)} < \epsilon.
\end{eqnarray*}
\end{prop}

\begin{proof}
Let $r>0$ and let $\{F_j\}_{j=1}^{\infty}$ and $\{G_j\}_{j=1}^{\infty}$ be the sets 
from 
Proposition~\ref{Covering} for this value of $r$. 
Let $f\in\bo$ have norm at most $1$ there holds:
\begin{align*}
(Tf)(z)&-\sum_{j=1}^{\infty}(M_{1_{F_{j}}}TPM_{1_{G_{j}}}f)(z)
    =\sum_{j=1}^{\infty}M_{1_{F_{j}}}\left(Tf-TPM_{1_{G_{j}}}f\right)(z)
\\&=\sum_{j=1}^{\infty}\sum_{i=1}^{\infty}1_{F_{j}}(z)
\ip{Tf-T1_{G_{j}}f}{K_ze_i}_{\bo}e_i
\\&=\sum_{j=1}^{\infty}\sum_{i=1}^{\infty}1_{F_{j}}(z)
\ip{f-1_{G_{j}}f}{T^{*}K_ze_i}_{\bo}e_i
\\&=\sum_{j=1}^{\infty}\sum_{i=1}^{\infty}1_{F_{j}}(z)
\ip{1_{G_{j}^c}f}{T^{*}K_ze_i}_{\bo}e_i
\\&=\sum_{j=1}^{\infty}\sum_{i=1}^{\infty}\int_{\Om}
1_{F_{j}}(z)1_{G_{j}^c}(w)\ip{f(w)}{T^{*}(K_ze_i)(w)}_{\h}d\sigma(w)e_i
\\&=\sum_{j=1}^{\infty}\sum_{i=1}^{\infty}\int_{\Om}
1_{F_{j}}(z)1_{G_{j}^c}(w)\sum_{k=1}^{\infty}f_k(w)
\ip{e_k}{T^{*}(K_ze_i)(w)}_{\h}d\sigma(w)e_i
\\&=\sum_{j=1}^{\infty}\sum_{i=1}^{\infty}\int_{\Om}
1_{F_{j}}(z)1_{G_{j}^c}(w)\sum_{k=1}^{\infty}f_k(w)
\ip{K_we_k}{T^{*}(K_ze_i)}_{\bo}d\sigma(w)e_i.
\end{align*}
Thus, we want to show that the integral operator with kernel given by:
\begin{align*}
\ip{M_r(z,w)e_k}{e_i}_{\h}
=\sum_{j=1}^{\infty}1_{F_{j}}(z)1_{G_{j}^c}(w)\abs{\ip{T^*K_ze_i}{K_we_k}_{\bo}}
\end{align*}
is bounded and that the operator norm goes to zero as $r\to\infty$.

Again we will use the Matrix Schur's Test (Lemma~\ref{MSchur}) with
\[\ip{M_r(z,w)e_k}{e_i}_{\h}=
\sum_{j=1}^{\infty}1_{F_{j}}(z)1_{G_{j}^c}(w)
\abs{\ip{T^*K_ze_i}{K_we_k}_{\bo}}, 
\hspace{0.15cm} h(z) \equiv \norm{\Kbz}_{\Bo}^{\al/2},\]
\[X = \Om, 
\hspace{0.15cm} d\mu(z) =  d\nu(z) = d\sigma(z).\] If $\kappa=0$ then set 
$\alpha=\frac{4-2\kappa}{4-\kappa}=1$. If $\kappa>0$ first choose $p_0$ such 
that $\frac{2}{p}<\frac{2}{p_0}<\frac{4-2\kappa}{4-\kappa}$ and denote by $q_0$ 
the conjugate of $p_0$, $q_0=\frac{p_0}{p_0-1}$. Then choose 
$\alpha\in (\frac{2}{p_0}, \frac{4-2\kappa}{4-\kappa})$ such that 
$q_0(\alpha-\frac{2}{p_0})<\kappa$. The condition $p>\frac{4-\kappa}{2-\kappa}$ 
ensures that such $p_0$ and $\alpha$ exist. 
Let $z\in\Om$ be arbitrary and fixed. Since $\{F_j\}_{j=1}^{\infty}$ 
forms a covering for $\Om$ there exists a unique $j$ such that $z\in F_j$. 
Note also that $D(z,r)\subset G_j$ so $G_j^{c}\subset D(z,r)^{c}$. There holds:
\begin{align*}
Q_1(r)
&:=\int_{\Om}\norm{K_w}_{\Bo}^{\al}\sum_{k=1}^{\infty}
    \ip{M_r(x,y)e_k}{e_i}_{\h}d\sigma(w)
\\&=\int_{\Om}\norm{K_w}_{\Bo}^{\al}\sum_{k=1}^{\infty}
   \sum_{j=1}^{\infty}1_{F_{j}}(z)1_{G_{j}^c}(w)
   \abs{\ip{T^*K_ze_i}{K_we_k}_{\bo}}d\sigma(w)
\\&=\int_{G_{j}^c}\norm{K_w}_{\Bo}^{\al}1_{F_{j}}(z)\sum_{k=1}^{\infty}
   \abs{\ip{T^*K_ze_i}{K_we_k}_{\bo}}d\sigma(w)
\\&\leq\int_{D(z,r)^c}\norm{K_w}_{\Bo}^{\al}1_{F_{j}}(z)\sum_{k=1}^{\infty}
   \abs{\ip{T^*K_ze_i}{K_we_k}_{\bo}}d\sigma(w)
\\&=\norm{K_z}_{\Bo}\int_{D(0,r)^c}\sum_{k=1}^{\infty}
   \abs{\ip{T^*k_ze_i}{k_{\vp_z(u)}e_k}_{\bo}}
   \norm{K_{\vp_z(u)}}_{\Bo-1}^{\al-1}d\lambda(u)
\\&\simeq\norm{K_z}_{\Bo}\int_{D(0,r)^c}\sum_{k=1}^{\infty}
   \abs{\ip{T^*U_z^*k_0e_i}{U_z^*k_{u}e_k}_{\bo}}
   \abs{\ip{k_z}{k_u}_{\Bo}}^{1-\al}d\lambda(u)
\\&=\norm{K_z}_{\Bo}\int_{D(0,r)^c}\sum_{k=1}^{\infty}
   \abs{\ip{T^{*z}k_0e_i}{k_{u}e_k}_{\bo}}
   \abs{\ip{k_z}{k_u}_{\Bo}}^{1-\al}d\lambda(u)
\\&=\norm{K_z}_{\Bo}^{\al}\int_{D(0,r)^c}\sum_{k=1}^{\infty}
   \abs{\ip{T^{*z}(k_0e_i)(u)}{e_k}_{\h}}
   \frac{\abs{\ip{K_z}{K_u}_{\Bo}}^{1-\al}}{\norm{K_u}^{2-\al}}d\lambda(u).
\end{align*}
Using H\"older's inequality we obtain that the last expression is no greater 
than
$$
\norm{\Kbz}_{\Bo}^{\al}
\lp\int_{D(0,r)^{c}}
    \lp\sum_{k=1}^{\infty}\abs{\ip{T^{*z}k_0e_i(u)}{e_k}_{\h}}\rp^{p_0}
    d\sigma(u)\rp^{\frac{1}{p_0}}
\left(\int_{\Om}\frac{\abs{\ip{\Kbz}{K_{u}}_{\Bo}}^{q_0(1-\al)}}
    {\norm{K_u}_{\Bo}^{q_0\left(2-\al-\frac{2}{p_0}\right)}}\,
    d\la(u)\right)^{\frac{1}{q_0}}.
$$
Let $r=q_0\left(2-\al-\frac{2}{p_0}\right)$ and $s=r-2q_0(1-\al)$. Then 
$r=\frac{p_0(2-\al-\frac{2}{p_0})}{p_0-1}=2-\frac{\al p_0}{p_0-1} > \kappa$ and
$s=q_0(\al-\frac{2}{p_0})<\kappa$ when $\kappa>0$ and $s=r>\kappa$ if
$\kappa=0$. This means that both $r$ and $s$ satisfy all conditions of 
\hyperref[A6]{A.6}. Thus, by Lemma~\ref{RF}, the second integral is bounded 
independent of $z$. Call this constant $C$. 

For the first integral, note that 
$\abs{\ip{T^{*z}k_0e_i(u)}{e_k}_{\h}}\simeq 
\abs{\ip{U_zT^*k_ze_i(u)}{e_k}_{\h}}$. Then using H\"{o}lder's Inequality 
there holds:
\begin{align*}
Q_1(r)
&\lesssim
    \norm{K_z}_{\Bo}^{\al}C
    \sigma(D(0,r)^c)^{\gamma}
    \lp\int_{\Om}
    \lp\sum_{k=1}^{\infty}\abs{\ip{U_zT^*k_ze_i(u)}{e_k}_{\h}}\rp^{p}
    d\sigma(u)\rp^{\frac{1}{p}}.
\end{align*}
Where $\gamma=1/p_0p'$ and $p'$ is conjugate exponent to $p$. 
Since $\sigma$ is a finite measure and since 
\begin{align*}
\sup_{i}\sup_{z\in\Om}
    \lp\int_{\Om}
    \lp\sum_{k=1}^{\infty}\abs{\ip{U_zT^*k_ze_i(u)}{e_k}_{\h}}\rp^{p}
    d\sigma(u)\rp^{\frac{1}{p}}
<\infty,
\end{align*}
$Q_1(r)$ goes to $0$ as $r\to\infty$. Thus, the first condition of 
Lemma~\ref{MSchur} is satisfied with a constant $o(1)$ as $r\to\infty$.

Next, we check the second condition. Fix $w\in\Om$. Let $J$ be a subset of all 
indices  $j$ such that $w\notin G_j$. If $z\in F_j$ for some $j\in J$, then
since $w\notin G_j$, there holds that $\m(w,F_j)>r$ and therefore 
$z$ is not in $D(w,r)$. Thus, $\cup_{j\in J}F_j \subset D(w,r)^{c}$ and 
consequently 
\begin{align*}
Q_2(r)
&:=\int_{\Om}\sum_{i=1}^{\infty}
    \norm{K_z}_{\Bo}\ip{M_r^*(z,w)e_i}{e_k}_{\h}d\sigma(z)
\\&=\int_{\Om}\sum_{i=1}^{\infty}
    \norm{K_z}_{\Bo}\ip{e_i}{M_r(z,w)e_k}_{\h}d\sigma(z)
\\&=\int_{\Om}\sum_{i=1}^{\infty}\norm{K_z}_{\Bo}
    \sum_{j=1}^{\infty}1_{F_{j}}(z)1_{G_{j}^c}(w)
    \abs{\ip{T^*K_ze_i}{K_we_k}_{\bo}}d\sigma(z)
\\&=\int_{\Om}\norm{K_z}_{\Bo}\sum_{i=1}^{\infty}
    \sum_{j=1}^{\infty}1_{F_{j}}(z)1_{G_{j}^c}(w)
    \abs{\ip{K_ze_i}{TK_we_k}_{\bo}}d\sigma(z)
\\&=\int_{\cup_{j\in J}F_j}\norm{K_z}_{\Bo}\sum_{i=1}^{\infty}
    \abs{\ip{K_ze_i}{TK_we_k}_{\bo}}d\sigma(z)
\\&\leq\int_{D(w,r)^{c}}\norm{K_z}_{\Bo}\sum_{i=1}^{\infty}
    \abs{\ip{K_ze_i}{TK_we_k}_{\bo}}d\sigma(z).
\end{align*}
Now, using the same estimates as above, but interchanging roles of 
$T$ and $T^*$ and $i$ and $k$ there holds:
\begin{align*}
Q_2(r)
&\lesssim
    \norm{K_w}_{\Bo}^{\al}C
    \sigma(D(0,r)^c)^{\gamma}
    \lp\int_{\Om}
    \lp\sum_{i=1}^{\infty}\abs{\ip{U_zTk_ze_k(u)}{e_i}_{\h}}\rp^{p}
    d\sigma(u)\rp^{\frac{1}{p}}.
\end{align*}
Thus, as before, $Q_2(r)$ goes to zero as $r\to\infty$. 

So both conditions of Matrix Schur's Test (Lemma~\ref{MSchur}) are satisfied
with constants that go to zero as $r\to\infty$. Thus, by choosing $r$ large 
enough, the integral operator with kernel given by $M_r(z,w)$ has operator
norm less than $\epsilon$. If $\{F_j\}_{j=1}^{\infty}$ and 
$\{G_j\}_{j=1}^{\infty}$ are the sets from
Proposition~\ref{Covering} associated to this valued of $r$, 
this also implies that:
\begin{eqnarray*} \norm{ T-\sum_{j=1}^{\infty}M_{1_{F_j} }
TPM_{1_{G_j} }}_{\l(\Om,\h;d\sigma)} < \epsilon.
\end{eqnarray*}
This proves the proposition for $\kappa>0$. When $\kappa=0$, the 
Proposition can be proven by making adaptations as in the proof of
Theorem ~\ref{RKT}.
\end{proof}

We now come to the main results of the section. 

\begin{thm}\label{RKTC} Let $T:\bo\to\bo$ be a linear operator and $\kappa$ be 
the constant from \hyperref[A6]{A.6}.
If
\begin{align}\label{e1}
\sup_{i}\sup_{z\in\Om}
\lp\int_{\Om}
    \lp\sum_{k=1}^{\infty}\abs{\ip{U_zT^*k_ze_i(u)}{e_k}_{\h}}\rp^{p}
    d\sigma(u)\rp^{\frac{1}{p}}
<\infty
\end{align}
and
\begin{align}\label{e2}
\sup_{k}\sup_{z\in\Om}
\lp\int_{\Om}
    \lp\sum_{i=1}^{\infty}\abs{\ip{U_zTk_ze_k(u)}{e_i}_{\h}}\rp^{p}
    d\sigma(u)\rp^{\frac{1}{p}}
<\infty
\end{align}
for some $p>\frac{4-\kappa}{2-\kappa}$,
and 
\begin{align}\label{smest}
\limsup_{d\to\infty}\norm{T M_{I_{(d)}}}_{\BL} = 0
\end{align}
then
\begin{itemize}
\item[(a)] $ \|T\|_e\simeq \sup_{\norm{f} \leq 1}
\limsup_{\m(z,0)\to\infty} \norm{T^zf}_{\bo}.$
\item[(b)] If $\sup_{e\in\Cd,\norm{e}_{\Cd}=1}
\lim_{\m(z,0)\to\infty}\norm{Tk_ze}_{\Bo}=0$ 
then $T$ must be compact.
\end{itemize}
\end{thm}

\begin{proof}
We first prove $(a)$. It is easy to deduce that
\begin{align}\label{eqn:lessthan}
\sup_{\norm{f}_{\bo} \leq 1}\limsup_{\m(z,0)\to\infty} 
\norm{T^zf}_{\bo}\mbox{}\lesssim \|T\|_e.
\end{align}
Indeed, using the triangle inequality and the fact that 
$\lim_{\m(z,0)\to\infty}\norm{A^zf}_{\Bo}=0$ for every compact operator $A$ 
(Lemma \ref{lm:Compactsaregood}) we
obtain that  
$$
\sup_{\norm{f}_{\bo} \leq 1}\limsup_{\m(z,0)\to\infty} 
\norm{T^zf}_{\bo}\leq \sup_{\norm{f}_{\bo} 
\leq1}\limsup_{\m(z,0)\to\infty} 
\norm{(T-A)^zf}_{\bo}\lesssim \norm{T-A}_{\BL}
$$ 
for any compact operator $A$. Now, since $A$ is arbitrary this immediately 
implies \eqref{eqn:lessthan}.

The other inequality requires more work. 
Proposition~\ref{MainEst1} and assumption~\eqref{smest} will play prominent roles.
Observe that the essential norm of $T$ as 
an operator in $\LL(\bo)$ is quasi-equal to the essential norm of $T$ as an 
operator in $\BL$. Therefore, it is enough to estimate the essential norm of $T$
as an operator on $\BL$. 

Let $\epsilon>0$ and fix a $d$ so large that
\begin{align*}
\norm{TM_{I_{(d)}}}_{\BL}<\epsilon.
\end{align*}
Then
\begin{align*}
\norm{T}_{e}
=\norm{T\ml + TM_{I_{(d)}}}_{e}
\leq\norm{T\ml}_{e}+\epsilon.
\end{align*}

By Proposition~\ref{MainEst1} there 
exists $r>0$ such that for the covering $\FF_r=\{F_j\}_{j=1}^{\infty}$ 
associated to $r$ 
\begin{eqnarray*} \norm{T\ml- \sum_{j=1}^{\infty}M_{1_{F_j} }
TPM_{1_{G_j} }\ml}_{\BL} < \epsilon.
\end{eqnarray*}

Note that by Lemma~\ref{ToeplitzCompact} the Toeplitz operators 
$PM_{1_{G_j}}\ml$ 
are compact. Therefore the finite sum 
$\sum_{j\leq m}M_{1_{F_j} }TPM_{1_{G_j}}\ml$
is compact for every $m,d\in \N$. So, it is enough to show that  
$$
\limsup_{m\to\infty}\norm{T_{(m,d)}}_{\BL}\lesssim \sup_{\norm{f} \leq 1}
\limsup_{\m(z,0)\to \infty} \norm{T^zf}_{\Bbt},
$$
where 
$$
T_{(m,d)}=  \sum_{j\geq m}M_{1_{F_j} }TPM_{1_{G_j} }\ml.
$$
Indeed,
\begin{align*}
\norm{T\ml}_{e}
&=\norm{T\ml P}_e
\\&\leq \norm{TP\ml-\sum_{j\leq m}M_{1_{F_j} }TPM_{1_{G_j}}\ml}_{\BL}
\\&\leq\epsilon + \norm{T_{(m,d)}}_{\BL}. 
\end{align*}
Of course, the implied constants should be independent of the truncation
parameter, $d$. 
Let $f\in\bo$ be arbitrary of norm no greater than $1$. There holds:

\begin{align*} 
\norm{T_{(m,d)} f}_{\BL}^2 
&= \sum_{j\geq m}\norm{M_{1_{F_j} }TP M_{1_{G_j} }\ml f}_{\bo}^2
\\&= \sum_{j\geq m} \frac{\norm{M_{1_{F_j} }TP M_{1_{G_j} } \ml f}_{\bo}^2}
{\norm{M_{1_{G_j} }\ml f}_{\bo}^2}\norm{M_{1_{G_j} }\ml f}_{\bo}^2
\\&\leq N\sup_{j\geq m}\norm{M_{1_{F_j} }T l_j}_{\bo}^2
\\&\leq N\sup_{j\geq m} \norm{T l_j}_{\bo}^2,
\end{align*}
where 
$$l_j:=\frac{PM_{1_{G_j} }\ml f}{\norm{M_{1_{G_j} }\ml f}_{\bo}}.$$
Therefore, 
$$\norm{T_{(m,d)}}_{\BL}\leq
N \sup_{j\geq m}\sup_{\norm{f}_{\bo}=1}\left\{\norm{T l_j}_{\bo}: l_j
=\frac{PM_{1_{G_j} }\ml f}{\norm{M_{1_{G_j} }\ml f}_{\bo}}\right\},$$ and hence
$$ \limsup_{m\to\infty}\norm{T_{(m,d)}}_{\BL}\leq N \limsup_{j\to\infty}
\sup_{\norm{f}_{\bo}=1}\left\{\norm{T g}: g=\frac{PM_{1_{G_j}}\ml f}
{\norm{M_{1_{G_j} }\ml f}_{\bo}}\right\}.$$

Let $\epsilon>0$. There exists a normalized sequence $\{f_j\}_{j=1}^{\infty}$ in 
$\bo$ such that 

$$ \limsup_{j\to \infty}\sup_{\norm{f}_{\bo}=1}\left\{\norm{Tg}: 
g=\frac{PM_{1_{G_j}}\ml f}{\norm{M_{1_{G_j} }\ml f}_{\bo}}\right\}-\epsilon\leq 
\limsup_{j\to\infty}\norm{Tg_j}_{\bo},$$ where
$$g_j:=\frac{PM_{1_{G_j} }\ml f_j}{\norm{M_{G_j}\ml f_j }_{\bo}}=
\frac{\int_{G_j}\sum_{k=1}^{d}
\ip{ f_j}{k_we_k}_{\bo}k_we_k\,d\la(w)}{\left(\int_{G_j}
\sum_{k=1}^{d}\abs{\ip{f_j}{k_we_k}_{\bo}}^2d\la(w)
\right)^{\frac{1}{2}}}.$$ 
It is clear that the functions $g_j$ are $d$--finite. 
Recall that $\abs{U^*_{z}k_w} \simeq
\abs{k_{\vp_{z}(w)}}$, and therefore, $U^*_{z}k_w = c(w,z)k_{\vp_{z}(w)}$, where
$c(w,z)$ is some function so that $\abs{c(w,z)}\simeq 1$.

There exists a $\rho>0$ such that if $z_j\in G_j$ then $G_j\subset D(z_j,\rho)$.
Thus, for each $j$, choose a $z_j$ in $G_j$. By a change of variables, there
holds

$$g_j=\int_{\vp_{z_j}(G_j)}a_j(\vp_{z_j}(w))U^*_{z_j}k_w\,d\la(\vp_{z_j}(w)),$$ 
where  
$$
a_j(w):=
\frac{\sum_{k=1}^{d}\ip{f_j}{k_we_k}_{\bo}e_k}{c(\vp_{z_j}(w), z_j)
\left(\int_{G_j}
\sum_{k=1}^{d}\abs{\ip{f_j}{k_we_k}_{\bo}}^2\,d\la(w)\right)^{\frac{1}{2}}}
$$ 
on $G_j$, and zero otherwise. 

We claim that $g_j=U^*_{z_j}h_j$, where  
$$h_j(z):=\int_{\vp_{z_j}(G_j)}a_j
(\vp_{z_j}(w))k_w(z)\,d\la(\vp_{z_j}(w)).
$$
First, by applying the integral form of Minkowski's inequality to the 
components of $h_j$, we conclude that each component is in 
$L^{2}(\Om,\h;\sigma)$ and therefore $h_j$ is also in $L^{2}(\Om,\h;\sigma)$,
and consequently in $\bo$. Now we need to show that for every $g\in
L^{2}(\Om,\h;\sigma)$ there holds $\ip{g_j}{g}_{L^{2}(\Om,\h;\sigma)}=
\ip{U_{z_j}^{*}h_j}{g}_{L^{2}(\Om,\h;\sigma)}=
\ip{h_j}{U_{z_j}g}_{L^{2}(\Om,\h;\sigma)}$. This is done by applying Fubini's
Theorem component--wise.  

For each $k=1,\ldots, d$, the total variation of each member of the 
sequence of measures
$\{\ip{a_j\circ\vp_{z_j}}{e_k}_{\h}d\lambda\circ\vp_{z_j}\}_{j=1}^{\infty}$, 
as elements in the
dual space of continuous functions on $(\overline{D(0,\rho)})$ satisfies 
$$\norm{\ip{a_j\circ\vp_{z_j}}{e_k}_{\h}
d\lambda\circ\vp_{z_j}}_{C(\overline{D(0,\rho)})^{*}}
\lesssim \lambda(D(0,\rho)).$$
Therefore, for each $k$, there exists a weak-$\ast$ convergent subsequence which
approaches some measure $\nu_k$. Let
\begin{align*}
h(z)=\sum_{k=1}^{d}\int_{D(0,\rho)}k_{w}(z)d\nu_k(w)e_k.
\end{align*}
Abusing notation, we continue to index the subsequence by $j$.
The weak-$\ast$ convergence implies that that $\ip{h_j}{e_k}_{\h}$ converges to 
$\ip{h}{e_k}_{\h}$ pointwise. 
By the Lebesgue Dominated Convergence 
Theorem, this implies that $\ip{h_j}{e_k}_{\h}$ converges to 
$\ip{h}{e_k}_{\h}$ in 
$L^{2}(\Om,\C;\sigma)$ and thus $\ip{h}{e_k}_{\h}$ is in 
$L^{2}(\Om,\C;d\sigma)$.
Since the $h_j$ and $h$ are $d$--finite, this implies that $h_{j}$ converges to
$h$ in $L^{2}(\Om,\h;d\sigma)$ and that $h\in L^{2}(\Om,\h;d\sigma)$.
Additionally,
$1\geq \norm{g_j}_{\bo}=\norm{U_{z_j}^{*}h_j}_{\bo}
\simeq \norm{h_j}_{\bo}$. So, $\norm{h}_{\bo}\lesssim 1$. 
So, finally, there holds: 
\begin{align*}
\limsup_{m\to\infty}\norm{T_m}_{\BL}
&\leq N \lim_{j\to\infty}\norm{Tg_j}_{L^{2}(\Om,\h;d\sigma)}+\epsilon 
\\&=N\lim_{j\to\infty} 
    \norm{TU^*_{z_j}h_j}_{L^{2}(\Om,\h;\sigma)}
    +\epsilon
\\&\leq N \limsup_{j\to\infty} 
    \norm{TU^*_{z_j}h}_{L^{2}(\Om,\h;\sigma)}+\epsilon
\\&\lesssim N  \limsup_{j\to\infty}
    \norm{T^{z_j}h}_{L^{2}(\Om,\h;\sigma)}+\epsilon.
\end{align*}
Again, the constants of equivalency do not depend on $d$.
Therefore, $$\limsup_{m\to\infty}\norm{T_m}_{\BL}\lesssim  
\sup_{\norm{f}_{L^{2}(\Om,\h;\sigma)} \leq 1}\limsup_{d(z,0)\to \infty} 
\norm{T^zf}_{L^{2}(\Om,\h;\sigma)}.$$

(b) 
Note that $\norm{T^zk_we}_{\bo}\simeq \norm{Tk_{\varphi_z(w)}e}_{\bo}$ and 
$d(\varphi_z(w),0)
\simeq d(w,z) \to \infty$ as $d(z,0)\to\infty$. 
Therefore, for all $w\in\Om$ and finite $e\in\h$ 
$\norm{T^zk_we}_{\bo}\to 0$ as $d(z,0)\to\infty$.  By density, this
implies that 
$\norm{T}_e
\simeq\sup_{\norm{f}_{L^{2}(\Om,\h;\sigma)} \leq 1}\limsup_{d(z,0)\to \infty} 
\norm{T^zf}_{L^{2}(\Om,\h;\sigma)}=0$. We are done.
\end{proof}

\begin{cor} 
\label{ToeCompact}
Let $\bo$ be a strong Bergman-type space for which $\kappa>0$. If $T$ is in the
Toeplitz algebra $\mathcal{T}_{\lf}$ then
\begin{itemize}
\item[(a)] $ \|T\|_e\simeq \sup_{\norm{f}_{\bo} \leq 1}
\limsup_{\m(z,0)\to\infty} 
\norm{T^zf}_{\bo}.$
\item[(b)] If $\sup_{e\in\Cd,\norm{e}_{\Cd}=1}
\lim_{\m(z,0)\to\infty}\norm{Tk_ze}_{\bo}=0$ then $T$ must be compact.
\end{itemize}
\end{cor}

\begin{proof}
We will show that $T$ satisfies the hypotheses of Theorem~\ref{RKTC}.
First, let 
$$
T=\sum_{k=1}^{M}\prod_{j=1}^{N}T_{u_{j,k}}
$$ 
where each $u_{j,k}
\in \lf$ and is $d_{j,k}$--finite. 
By the
triangle inequality, is suffices to show that $T$ satisfies the hypotheses
of Theorem~\ref{RKTC} when
$T=\prod_{j=1}^{N}T_{u_{j}}$ and $u_j\in \lf$ and
$u_j$ is $d_j$--finite. Clearly, $T$ satisfies \eqref{smest}. Now we will show 
that it also satisfies \eqref{e1} and \eqref{e2}.
For any $z\in\Om$ and $i,k\in\N$ there holds
\begin{align*}
\ip{U_{z}Tk_ze_k}{e_i}_{\h}
&=\ip{\left(\prod_{j=1}^{N}T_{u_j\circ\vp_{z}}\right)(k_0e_k)}{e_i}_{\h}
\\&=\ip{\left(\prod_{j=1}^{N}PM_{u_j\circ\vp_{z}}\right)(k_0e_k)}{e_i}_{\h}.
\end{align*}
By the boundedness of the Bergman projection and the finiteness of the symbols,
we deduce that $T$ satisfies \eqref{e1}. The same argument shows that 
$T$ satisfies \eqref{e2}.

Now, let $T$ be a general operator in $\mathcal{T}_{\lf}$. Note that if
we prove that (a) holds, then it follows easily that $(b)$ also holds.
In the proof of Theorem \ref{RKTC}, we proved that
\begin{align*}
\sup_{\norm{f}_{\bo} \leq 1}\limsup_{\m(z,0)\to\infty}
\norm{T^zf}_{\bo}\lesssim\|T\|_e.
\end{align*}
So, we only need to prove to other inequality.
Since $T\in\mathcal{T}_{\lf}$, there is
an operator, $S$, that is a finite sum of finite products of Toepltitz operators
with $\lh$ symbols and 
such that $\norm{T-S}_{e}\leq\norm{T-S}_{\mathcal{L}(\bo)}<\epsilon$.
Using the fact that (a) is true for operators of this form, there holds:
\begin{align*}
\norm{T}_{e}\leq \norm{T-S}_{e}+\norm{S}_e
\leq \epsilon + \sup_{\norm{f}_{\bo} \leq 1}
    \limsup_{\m(z,0)\to\infty}\norm{S^zf}_{\bo}.
\end{align*}
If $\norm{f}_{\bo}\leq 1$, there holds:
\begin{align*}
\norm{S^{z}f}_{\bo}\leq\norm{(T-S)^{z}f}_{\bo} + \norm{T^{z}f}_{\bo}
\lesssim \epsilon + \norm{T^{z}f}_{\bo}.
\end{align*}
Combining the last two inequalities we obtain the desired inequality for $T$.
\end{proof}

Our next goal is to show that the previous corollary holds even with a weaker 
assumption. Let $\textnormal{BUCO}$ denote the algebra of 
operator--valued functions
$u:(\Om,\m)\to(\mathcal{L}(\h),\norm{\cdot}_{\mathcal{L}(\h)})$ that are
in $\lf$ and uniformly continuous. This is equivilent to requiring that
$u$ be $d$--finite and requiring that 
$\ip{ue_k}{e_i}_{\h}\in\textnormal{BUC}(\Om,\m)$ for $i,k=1,\cdots d$,  where 
$\textnormal{BUC}(\Om,\m)$ is the algebra of bounded uniformly continuous 
functions on $\Om$. Let 
$\mathcal{T}_\textnormal{BUCO}$ denote the algebra generated by the Toeplitz 
operators with symbols from $\textnormal{BUCO}$.

In this section we show that if $T\in \mathcal{T}_{\textnormal{BUCO}}$ and
$\ip{Tk_ze_k}{k_ze_i}_{\bo}\to 0$ as 
$\m(z,0)\to\infty$, for every $i,k\in\N$ then $T$ is compact.   
Recall that for a
given operator $T$ the Berezin transform of $T$ is an operator--valued 
function on $\Om$ given by the formula
$$
\ip{\tilde{T}(z)e_k}{e_i}_{\h}:=\ip{T\kbz e_k}{\kbz e_i}_{\bo}.
$$

\begin{thm}
\label{Berezin} Let $T\in\mathcal{T}_\textnormal{BUCO}$. Then  
$$\limsup_{\m(z,0)\to\infty}\ip{\widetilde{T}(z)e_k}{e_i}_{\h}=0
$$ 
if and only if 
$$
\limsup_{\m(z,0)\to\infty} \norm{T^zf}_{\bo}=0
$$ 
for every $f\in\Bbt$. In particular, if the Berezin transform $\tilde{T}(z)$ ``vanishes at the boundary of $\Om$'' then the operator $T$ must be compact.
\end{thm}
For the remainder of this section, SOT will denote the strong operator topology 
in $\mathcal{L}(\bo)$ and WOT will denote the weak operator topology in $\mathcal{L}(\bo)$.

The key to proving Theorem \ref{Berezin} will be the following two lemmas.

\begin{lm}
The Berezin transform is one to one. That is, if $\widetilde{T}=0$, then
$T=0$.
\end{lm}
\begin{proof}
Let $T\in\mathcal{L}(\bo)$ and suppose that
$\widetilde{T}=0$. 
Then there holds:
\begin{align*}
0=\ip{T(k_ze_k)}{k_ze_i}_{\bo}
=\frac{1}{K(z,z)}\ip{T(K_ze_k)}
{K_ze_i}_{\bo}
\end{align*}
for all $z\in\Om$ and for all $i,k\in\N$. In particular, there holds:
\begin{align*}
\frac{1}{K(z,z)}\ip{T(K_ze_k)}
{K_ze_i}_{\bo}\equiv 0.
\end{align*}
Consider the function 
\begin{align*}
F(z,w)=\ip{T(K_we_k)}{K_ze_i}_{\bo}.
\end{align*}
This function is analytic in $z$, conjugate analytic in $w$ and $F(z,z)=0$ for 
all $z\in\Om$. By a standard result for several complex variables (see for 
instance \cite{kra}*{Exercise 3 p. 365}) this implies 
that $F$ is identically $0$. Using the reproducing property, we conclude that
\begin{align*}
F(z,w)=\ip{T(K_we_k)(z)}{e_i}_{\h}\equiv 0,
\end{align*}
and hence 
\begin{align*}
T(K_we_k)(z)\equiv 0,
\end{align*}
for every $w\in\Om$ and $k\in\N$. 
Since the products $K_we_k$ span $\bo$, we conclude that 
$T\equiv 0$ and the desired result follows.
\end{proof}

\begin{lm}\label{SOT} Let $u\in\textnormal{BUCO}$.  For any sequence 
$\{z_n\}_{n=1}^{\infty}$ in $\Omega$, the sequence of Toeplitz operators 
$T_{u\circ \varphi_{z_n}}$ has a \textnormal{SOT} convergent subnet.
\end{lm}

\begin{proof}
Since $u\in\textnormal{BUCO}$, it is finite and 
$\ip{Te_k}{e_j}_{\h}\in \textnormal{BUC}(\Om,\m)$. The result therefore follows
easily from the corresponding scalar--valued case, 
\cite{MW2}*{Lemma 4.7} by taking limits ``entry--wise''. 
\end{proof}

\begin{proof}[Proof of Theorem~\ref{Berezin}]   Suppose that  
$\limsup_{\m(z,0)\to\infty} \norm{T^zf}_{\bo}=0$ for every $f\in\bo$. 
Take
$f\equiv e_k$ then $\norm{T^{z}e_k}_{\bo}^{2}\simeq\norm{Tk_{z}e_k}_{\bo}^{2}$. 
Then there holds that:
\begin{align*}
\abs{\ip{\widetilde{T}(z)e_k}{e_i}_{\h}}
=\abs{\ip{Tk_ze_k}{k_ze_i}_{\mathcal{B}(\Om)}}
\leq\norm{Tk_{z}e_k}_{\mathcal{B}(\Om)}.
\end{align*}
Therefore, $\limsup_{\m(z,0)\to\infty}\ip{\widetilde{T}(z)e_k}{e_i}_{\h}=0$
for all $i,k\in\N$.

In the other direction, suppose that  $\lim_{\m(z,0)\to\infty} 
\abs{\ip{\tilde{T}(z)e_k}{e_i}_{\h}}=0$ for every $i,k\in\N$ but  
$\limsup_{\m(z,0)\to\infty} \norm{T^zf}_{\bo}>0$ for some $f\in\bo$. In this case there exists a sequence $\{z_n\}_{n=1}^{\infty}$ 
with $\m(z_n,0)\to\infty$  such that 
$\norm{T^{z_n}f}_{\bo}\geq c>0.$ We will show that $T^{z_n}$ has a subnet
that converges to the zero operator in SOT. This, of course, will be a 
contradiction.
Observe first that $T^{z_n}$ has a subnet which converges in the WOT. Call this
operator $S$. Slightly abusing notation, we continue to denote the subnet by $\{z_n\}_{n=1}^{\infty}$. Then $\ip{{T}^{z_n}k_ze_k}{k_ze_i}_{\bo}\to \ip{{S}e_k}{e_i}_{\h}$ for every $i,k\in\N$. Thus, the entries of $\widetilde{T}$ converge pointwise
to the entries of $\widetilde{S}$. More precisely, for every $z\in\Om$ and
for every $k,i\in\N$, there holds 
$\ip{\widetilde{T}^{z_n}(z)e_k}{e_i}_{\bo}\to 
\ip{\widetilde{S}(z)e_k}{e_i}_{\h}$. The assumption 
$\lim_{\m(z,0)\to\infty} \abs{\ip{\tilde{T}(z)e_k}{e_i}_{\h}}=0$
implies that $\ip{\tilde{T}^{z_n}e_k}{e_i}_{\h}\to 0$ pointwise for every 
$i,k\in\N$ as well and hence $\tilde{S}\equiv 0$. Therefore $S$ is the zero operator and consequently $T^{z_n}$ converges to zero in the WOT.

Next, we use the fact that $T$ is in $\mathcal{T}_{\textnormal{BUCO}}$ to show that there exists a subnet of $T^{z_n}$ which converges in SOT. Let $\epsilon>0$, then there exists an operator $A$ which is a finite sum of finite products of Toeplitz operators with symbols in $\textnormal{BUCO}$ such that $\norm{T-A}_{\mathcal{L}(\bo)}<\epsilon$.  We first show that $A^{z_n}$ must have a convergent subnet in SOT. By linearity we can consider only the case when $A=T_{u_1}T_{u_2}\cdots T_{u_k}$ is a finite product of Toeplitz operators. As noticed before $$A^{z_n}=T_{u_1\circ \varphi_{z_n}}T_{u_2\circ \varphi_{z_n}}\cdots T_{u_k\circ \varphi_{z_n}}.$$ Now, since a product of SOT convergent nets is SOT convergent, it is enough to treat the case when $A=T_{u}$ is a single Toeplitz operator. But, the single Toeplitz operator case follows directly from Lemma~\ref{SOT}. 

Denote by $B$ the SOT limit of this subnet $A^{{z_n}_k}$. If $f\in\bo$ is
of norm at most $1$, there holds:
\begin{align*}
\norm{Bf}_{\bo}^{2}
&=\ip{Bf}{Bf}_{\bo}
\\&\leq \abs{\ip{Bf-A^{{z_n}_k}f}{Bf}_{\bo}} +
\abs{\ip{T^{{z_n}_k}f-A^{{z_n}_k}f}{Bf}_{\bo}} +
\abs{\ip{T^{{z_n}_k}f}{Bf}_{\bo}}.
\end{align*}
Using the fact that $B$ is the SOT limit of bounded operators, we deduce that
$B$ is bounded. By the weak convergence of $T^{z_n}$, by taking $n_k$ ``large'' 
enough, the outer terms above can be made less than $\epsilon$. By assumption, 
the middle term is less than $\epsilon$. We deduce 
that $\norm{B}_{\mathcal{L}(\bo}\lesssim \epsilon$.
Now, for every $f\in\bo$ of norm no greater than $1$ there holds 
\begin{align*} 
\norm{T^{{z_n}_k}f}_{\bo}\leq 
\norm{A^{{z_n}_k}f}_{\bo}+\norm{(T^{{z_n}_k}-A^{{z_n}_k})f}_{\bo}.
\end{align*}
Therefore, $\limsup \norm{T^{{z_n}_k}f}_{\bo}\lesssim \norm{Bf}_{\bo}+
\epsilon\leq 2\epsilon$. Finally, the fact that $\epsilon>0$ was arbitrary implies that $\lim \norm{T^{{z_n}_k}f}_{\bo}=0$ for all $f$. Consequently, we found a subnet $T^{{z_n}_k}$ which converges to the zero operator in SOT. We are done.
 
\end{proof}

\section{Density of $\mathcal{T}_{\lf}$}\label{dens}
In this section, we will prove that for a restricted class of 
Bergman--type function spaces, that an operator is compact if and only if
it is in the Toeplitz algebra and its Berezin transform vanishes on 
$\partial\Om$. See, for example, \cites{Sua,MSW,MW,MW2,RW,BI} for 
similar results for the scalar--valued Bergman--type spaces. 

First, let $\mu$ be a measure on $\Om$. We define the Toeplitz operator on
$\Bo$ with symbol $\mu$ by:
\begin{align*}
(T_{\mu}f)(z)=\int_{\Om}\ip{K_w}{K_z}_{\Bo}f(w)d\mu(w).
\end{align*}
Recall that a positive measure $\mu$ on $\Om$ is said to be Carleson with 
respect to $\sigma$ if there is a $C$ such that for every $f\in\Bo$ there holds:
\begin{align*}
\int_{\Om}\abs{f}^{2}d\mu \leq C\int_{\Om}\abs{f}^{2}d\sigma. 
\end{align*}
Clearly, if $a$ is a bounded function on $\Om$, then $ad\sigma$ is Carleson
with respect to $\sigma$. 

Next, let $\mu$ be a countably additive matrix--valued function from the Borel 
sets of
$\Om$ to $\l(\h)$ such that $\mu(\emptyset)=0$. Then we say that $\mu$ is
a matrix--valued measure. The entries of $\mu$, which are given
by $\ip{\mu e_k}{e_j}_{\h}$,
are all measures on $\Om$. We can define a Toeplitz operator $T_{\mu}$ on
$\bo$ by the formula:
\begin{align*}
(T_{\mu}f)(z)=\int_{\Om}\ip{K_w}{K_z}_{\Bo}d\mu(w)f(w).
\end{align*}

For this section, we define a more restrictive $\h$--valued Bergman--type space. 
We add the additional assumption:
\begin{itemize}

\item[\label{A.8} A.8] If $\mu$ is a scalar--valued measure on $\Om$ 
whose total variation is
Carleson with respect to $\sigma$, then 
$T_{\mu}\in\mathcal{T}_{\textnormal{BUC}}$, where
$\mathcal{T}_{\textnormal{BUC}}$ is the algebra of operators on $\Bo$ 
generated by Toeplitz operators with symbols that are bounded and 
uniformly continuous on $\Om$. 
\end{itemize}
We will call such spaces $\Boa$ and we will call their $\h$--valued extensions
$\boa$. (The $\mathcal{A}$ is for ``approximation''.) This
is not a trivial assumption and it is (at this point) not known whether this holds
for all Bergman--type spaces (see \cite{MW2}). 
It does hold in the standard Bergman spaces on the ball 
and polydisc and also on the Fock space see \cites{Sua,MSW,MW,BI}. Thus, the 
following theorem can be viewed as an extension of the main theorems in 
\cites{Sua,MSW,MW,BI} to the $\h$--valued setting. 

We will prove the following theorem:
\begin{thm}
\label{BerezinToe} Let $T\in\mathcal{L}(\boa)$. Then $T$ is compact if and only
if $\limsup_{\m(z,0)\to\infty} \tilde{T}(z)=0$ and $T\in\mathcal{T}_{\lf}$.
\end{thm}

We will first prove the following lemma. The proof of the following lemma uses
Assumption \hyperref[A.8]{A.8}.
\begin{lm}\label{toeq}
On $\boa$, 
\begin{align*}
\mathcal{T}_{\textnormal{BUCO}}=\mathcal{T}_{\lf}.
\end{align*}
\end{lm}
\begin{proof}
It is clear that $\mathcal{T}_{\textnormal{BUCO}}\subset\mathcal{T}_{\lf}$.
Now, let $u\in \lf$. Since $\ip{u e_k}{e_j}_{\h}$ is bounded, for every 
$k,j\in\N$ it follows that $\abs{\ip{u e_k}{e_j}_{\h}}d\sigma$ is Carleson with
respect to $\sigma$ for every $k,j\in\N$. So then the Toeplitz operator
on $\Boa$ with symbol $\abs{\ip{u e_k}{e_j}_{\h}}d\sigma$ is in 
$\mathcal{T}_{\textnormal{BUC}}$. Since $u$ is a finite symbol, it easily
follows that $T_{u}\in\mathcal{T}_{\textnormal{BUCO}}$. 
Thus $\mathcal{T}_{\lf}\subset\mathcal{T}_{\textnormal{BUCO}}$. This completes 
the proof.
\end{proof}

\begin{proof}[Proof of Theorem \ref{BerezinToe}]
If $T\in\mathcal{T}_{\lf}$ then the previous lemma shows that 
$T\in\mathcal{T}_{\textnormal{BUCO}}$. So
if $\limsup_{d(z,0)\to\infty} \tilde{T}(z)=0$ then by Theorem 
\ref{Berezin}, $T$ is compact. On the other hand, if $T$ is compact, 
then by Lemma \ref{compact} there holds that 
$\limsup_{d(z,0)\to\infty} \tilde{T}(z)=0$. So, we only need to show that
$T$ is in $\mathcal{T}_{\lf}$. Since $T$ is compact, it suffices to show that
each rank $1$ operator is in $\mathcal{T}_{\lf}$. The rank $1$ operators
are given by the formula:
\begin{align*}
(f\otimes g)(h) = \ip{h}{g}_{\bo}f.
\end{align*}
So, we need to show that $f\otimes g\in\mathcal{T}_{\lf}$. 
Let $p_f$ be a polynomial such that 
$\norm{f-p_f}_{\boa}<\frac{\epsilon}{\norm{g}_{\boa}}$ and
$p_g$ a polynomial such that 
$\norm{g-p_g}_{\boa}<\frac{\epsilon}{\norm{p_f}_{\boa}}$. Then for
$h\in\boa$ there holds:
\begin{align*}
\norm{(f\otimes g)h-(p_f\otimes p_g)h}_{\boa}
&=\norm{\ip{h}{g}_{\boa}f-\ip{g}{p_g}_{\boa}p_f}_{\boa}
\\&\leq\norm{\ip{h}{g}_{\boa}(f-p_f)}_{\boa}
    +\norm{\ip{h}{g-p_g}_{\boa}p_f}_{\boa}
\\&\leq 2\epsilon\norm{h}_{\boa}.
\end{align*}
Therefore, if we can show that $p_f\otimes p_g\in\mathcal{T}_{\lf}$, then
we will be finished. For the following computation, we use the following
notational conviniencies. Let $E_{i,j}$ be the matrix such that
$\ip{E_{i,j}e_k}{e_l}=1$ when $k=j$ and $l=i$ and zero otherwise. That is,
$E_{i,j}$ is the matrix with a $1$ in the $(i,j)$ position and zeros 
everywhere else. 
We abuse notation and write $f$ in place of $p_f$ and $g$ in place of 
$p_g$, keeping in mind that this means that $f$ and $g$ are both now 
polynomials. 
Lastly, we will abuse notation again and $P$ will also denote the projection 
from $L^{2}(\Om,\C;d\sigma)$ onto $\Boa$. Observe that if $f\in\boa$ then:
\begin{align*}
Pf=\sum_{i=1}^{\infty}(Pf_i)e_i,
\end{align*}
where on the left hand side, $P$ is the projection on $L^{2}(\Om,\h;d\sigma)$
and on the right hand side $P$ is the projection on $L^{2}(\Om,\C;d\sigma)$. 
Observe that since $K_0(z)=\norm{K_0}_{\Boa}k_0(z)$ and since $k_0\equiv 1$ on $\Om$, 
there holds that $K_0\equiv\norm{K_0}_{\Boa}$. 

Using these facts, we compute:
\begin{align*}
\sum_{i=1}^{\infty}\sum_{k=1}^{\infty}\left(T_{f_iE_{i,i}}
    T_{\delta_{0}E_{i,i}}
    T_{\overline{g_k}E_{i,i}}T_{E_{i,k}}h\right)(z)
&=\sum_{i=1}^{\infty}\sum_{k=1}^{\infty}\left(f_i
    T_{\delta_{0}E_{i,i}}
    T_{\overline{g_k}E_{i,i}}T_{E_{i,k}}h\right)(z)
\\&=\sum_{i=1}^{\infty}\sum_{k=1}^{\infty}f_i(z)
    \left(T_{\overline{g_k}E_{i,i}}T_{E_{i,k}}h\right)(0)
\\&=\sum_{i=1}^{\infty}\sum_{k=1}^{\infty}f_i(z)P
    \left(\overline{g_k}h_{k}e_i\right)(0)
\\&=\sum_{i=1}^{\infty}f_i(z)\sum_{k=1}^{\infty}\int_{\Om}h_k(w)
    \overline{g_k(w)}\overline{K_0(w)}d\sigma e_i
\\&=\norm{K_0}_{\Boa}\sum_{i=1}^{\infty}f_i(z)\int_{\Om}
    \sum_{k=1}^{\infty}h_k(w)
    \overline{g_k(w)}d\sigma e_i
\\&=\norm{K_0}_{\Boa}\ip{h}{g}f(z) 
=\norm{K_0}_{\Boa}(f\otimes g)h(z).
\end{align*}
We therefore conclude that:
\begin{align*}
\sum_{i=1}^{\infty}\sum_{k=1}^{\infty}\left(T_{f_iE_{i,i}}
    T_{\norm{K_0}_{\Boa}^{-1}\delta_{0}E_{i,i}}
    T_{\overline{g_k}E_{i,i}}T_{E_{i,k}}h\right)(z)
=(f\otimes g)h(z).
\end{align*}
Since pointwise evaluation is a bounded linear functional, we conclude that 
$\norm{K_0}_{\Boa}^{-1}\delta_0$ is a Carleson measure for $\Boa$ with respect to 
$\sigma$. Thus, $T_{\norm{K_0}_{\Boa}^{-1}\delta_0E_{i,i}}\in\mathcal{T}_{\lf}$ 
for every $i\in\N$.
Furthermore, each of the operators 
$T_{f_iE_{i,i}}$, $T_{\overline{g_k}E_{i,i}}$, and $T_{E_{i,k}}$ are 
Toeplitz operators with symbols in $\textnormal{BUCO}$ and so each one is
in $\mathcal{T}_{\lf}$. Since $f_i$ and $g_i$ are finite, the sums above
are finite. This implies that the operator given by the formula
\begin{align*}
\sum_{i=1}^{\infty}\sum_{k=1}^{\infty}\left(T_{f_iE_{i,i}}
    T_{\norm{K_0}_{\Boa}^{-1}\delta_{0}E_{i,i}}
    T_{\overline{g_k}E_{i,i}}T_{E_{i,k}}h\right)
\end{align*}
is a member of $\mathcal{T}_{\lf}$ and therefore, $(f\otimes g)$ is in 
$\mathcal{T}_{\lf}$ for all polynomials, $f$ and $g$. This completes the proof.
\end{proof}

\section{Acknowledgements} 
The author would like to thank Michael Lacey for supporting him as a
research assistant for the Spring semester of 2014 (NSF DMS grant \#1265570)
and Brett Wick for supporting him as a research
assistant for the Summer semester of 2014 (NSF DMS grant \#0955432) and for discussing
the problem with him.

\end{document}